\documentclass[a4paper,fleqn]{cas-sc}

\numberwithin{equation}{section}

\usepackage{amsthm}

\newtheorem{definition}{Definition}[section]
\newtheorem{remark}{Remark}[section]
\newtheorem{theorem}{Theorem}[section]
\newtheorem{lemma}{Lemma}[section]
\newtheorem{corollary}{Corollary}[section]
\newtheorem{assumption}{Assumption}[section]
\newtheorem{proposition}{Proposition}[section]
\newtheorem{example}{Example}[section]

\usepackage[authoryear,longnamesfirst]{natbib}
\usepackage[nameinlink, noabbrev, capitalise]{cleveref}
\usepackage[group-digits=false]{siunitx}
\usepackage{booktabs}
\usepackage{enumitem}
\usepackage{tcolorbox}

\usepackage{subcaption}
\def\tsc#1{\csdef{#1}{\textsc{\lowercase{#1}}\xspace}}
\tsc{WGM}
\tsc{QE}
\newcommand{\N}{\mathbb N}
\newcommand{\R}{\mathbb R}

\newcommand{\cC}{\mathcal C}
\newcommand{\cF}{\mathcal F}
\newcommand{\cH}{\mathcal H}

\newcommand{\bE}{\mathbb E}

\newcommand{\eps}{\varepsilon}
\newcommand{\cN}{\mathcal N}
\newcommand{\cS}{\mathcal S}

\newcommand{\nN}{\cN\cN}
\newcommand{\sNN}{\cS\cN\cN}
\newcommand{\Lip}{\text{Lip}}
\newcommand{\E}{\mathbb E}

\DeclareMathOperator*{\argmin}{arg\,min}
\DeclareMathOperator{\relu}{ReLU}

\newtcolorbox{boxedremark}[1][]{
  colback=gray!5!white,
  colframe=blue!40!black,
  fonttitle=\bfseries,
  title=Remark,
  #1
}

\begin{document}
\let\WriteBookmarks\relax
\def\floatpagepagefraction{1}
\def\textpagefraction{.001}

\shorttitle{Approximation bounds for norm constrained deep neural networks}    

\shortauthors{F.P. Maiale, A. Trofimova, A. De Marinis}  

\title [mode = title]{Approximation bounds for norm constrained deep neural networks}  



%

\author[1]{Francesco Paolo Maiale}[orcid=0000-0002-3389-6937]
\ead{francescopaolo.maiale@gssi.it}
\author[1]{Anastasiia Trofimova}[orcid=0000-0001-9119-133X]
\ead{anastasiia.trofimova@gssi.it}
\author[2]{Arturo De~Marinis}[orcid=0009-0004-4250-1054]
\ead{A.DeMarinis@tudelft.nl}

\cormark[1]
\cortext[1]{Corresponding author}




\affiliation[1]{organization={Gran Sasso Science Institute, Division of Mathematics},
            addressline={Viale Francesco Crispi 7}, 
            city={L'Aquila},
            postcode={67100}, 
            country={Italy}}
\affiliation[2]{organization={Delft University of Technology, Faculty of Aerospace Engineering},
            addressline={Kluyverweg 1}, 
            city={Delft},
            postcode={2629HS}, 
            country={Netherlands}}










\begin{abstract}
This paper studies the approximation capacity of neural networks with an arbitrary activation function and with norm constraint on the weights. Upper and lower bounds on the approximation error of these networks are computed for smooth function classes. The upper bound is proven by first approximating high-degree monomials and then generalizing it to functions via a partition of unity and Taylor expansion. The lower bound is derived through the Rademacher complexity of neural networks. A probabilistic version of the upper bound is also provided by considering neural networks with randomly sampled weights and biases. Finally, it is shown that the assumption on the regularity of the activation function can be significantly weakened without worsening the approximation error, and the approximation upper bound is validated with numerical experiments.
\end{abstract}



\begin{keywords}
Approximation theory \sep Deep learning \sep Neural networks \sep \vspace{2.5mm} \textit{MSC}: 41A25, 41A29, 68T07
\end{keywords}

\maketitle

\section{Introduction}

Neural networks have emerged as powerful tools for approximating complex functions across various scientific and engineering disciplines. They are parametric maps obtained by composing affine and nonlinear functions. We define a neural network $\varphi: \R^d \rightarrow \R^n$ with $L\in\N$ layers as follows:
\begin{equation}\label{eq:nn1}
    \phi(x) = \ell_L \circ \sigma_{L-1} \circ \ell_{L-1} \circ \dots \circ \sigma_0 \circ \ell_0(x), \qquad x\in\R^d,
\end{equation}
where $\sigma_0, \ldots, \sigma_{L-1}:\mathbb{R}\to\mathbb{R}$ are nonlinear and non-polynomial scalar activation functions applied entrywise, while $\ell_0,\dots,\ell_L$, called layers, are affine functions defined as
\[
    \ell_i(u) = A_i u + b_i, \qquad u\in\R^{m_i},\ A_i\in\R^{n_i\times m_i},\ b_i\in\R^{n_i},\quad  i=0,\dots,L,
\]
with $m_0=d$, $m_{i+1}=n_i$, $i=0,\dots,L-1$, and $n_L=n$. The matrices $A_i$ and vectors $b_i$, $i=0,\dots,L$, are called \emph{weights} and \emph{biases} respectively, and each component of $\sigma_i\circ\ell_i(u)$, $i=0,\dots,L-1$, and $\ell_L(u)$ is called a \emph{neuron} or \emph{unit}. We also refer to the composition $\sigma_i\circ\ell_i(u)$ as a layer. We denote by $W := \max\{n_0,\dots,n_L\}$ the \emph{width} of the neural network and we call $L$ its \emph{depth}. We distinguish two cases:
\begin{itemize}
    \item if $L\le1$, we call $\varphi$ a shallow neural network,
    \item if $L>1$, we call $\varphi$ a deep neural network.
\end{itemize}
Additionally, we note that neurons of the same layer are also allowed to have different activation functions. This is needed in the proof of some results throughout the paper. 

The following definition clarifies what we mean by \emph{approximation}.

\begin{definition}
    Let $d,n\in\N$ and $\|\cdot\|$ be a norm on $\R^n$. A set of functions $\cH$ mapping from $\R^d$ to $\R^n$ is a \emph{universal approximator} for the space $\cC(\R^d,\R^n)$ of continuous functions with respect to the topology of compact convergence if, for any $f\in \cC(\R^d,\R^n)$, any compact set $K\subset\R^d$, and any $\eps>0$, there exists a function $\varphi_\eps \in \cH$ such that
    \begin{equation}\label{eq:app_err}
        \|f-\varphi_\eps\|_{\infty, K} := \max_{x\in K} \|f(x)-\varphi_\eps(x)\| \le \eps.
    \end{equation}
\end{definition}

The approximation capabilities of neural networks have been extensively studied over the past several decades and continue to attract significant research interest. This section provides a brief review of the relevant literature.

\cite{cybenko1989approximation} and~\cite{hornik1989multilayer} showed that shallow neural networks with arbitrary width and sigmoidal activation functions -- namely, increasing functions $\sigma:\R\to\R$ with $\lim_{x\to-\infty}\sigma(x)=0$ and $\lim_{x\to\infty}\sigma(x)=1$ -- are universal approximators for continuous functions.~\cite{hornik1991approximation} revealed that this property arises not from the specific choice of activation function but from the compositional structure of the network itself.~\cite{leshno1993multilayer} and~\cite{pinkus1999approximation} showed that shallow neural networks with arbitrary width are universal approximators if and only if their activation function is not a polynomial.

The recent breakthrough of deep learning has stimulated much research on the approximation ability of neural networks with arbitrary depth.~\cite{gripenberg2003approximation} and~\cite{yarotsky2017error,lu2017expressive, hanin2017approximating} studied the universal approximation property of deep neural networks with arbitrary depth and bounded width using ReLU activation functions.~\cite{kidger2020universal} extended those results to arbitrarily deep neural networks with bounded width and arbitrary activation functions.

\cite{maiorov1999lower} studied networks of bounded depth and width. They showed that there exists an analytic sigmoidal activation function such that neural networks with two hidden layers and bounded width are universal approximators.~\cite{guliyev2018approximation} designed a family of two hidden layer neural networks with fixed width and a specific sigmoidal activation function that is a universal approximator of continuous functions.~\cite{guliyev2018on} also constructed shallow neural networks with bounded width that are universal approximators for univariate functions. However, this construction does not apply to multivariate functions.~\cite{shen2022optimal} computed precise quantitative information on the depth and width required to approximate a target function by deep and wide ReLU neural networks.

The results in most of the mentioned papers are only existence results, that is, given $f\in \cC(\R^d,\R^n)$, $K\subset\R^d$ compact, and $\eps>0$, there exists a neural network $\varphi_\eps$ such that~\eqref{eq:app_err} holds; however, they do not provide any information
\begin{itemize}
    \item on how to choose the number of layers of $\varphi_\eps$ and the dimension of the weights,
    \item on how to compute exactly the weights of $\varphi_\eps$,
    \item on the width or depth required for $\varphi_\eps$ to achieve the target error tolerance $\eps$.
\end{itemize}
This gap has motivated recent work on \emph{approximation rates}, which quantify how the approximation error decays as a function of the network's depth or width.

Approximation rates of shallow neural networks have been computed in~\cite{barron1993universal}. Recent works have also derived approximation rates of shallow neural networks with particular activation functions, for example the ReLU and cosine activation functions in~\cite{siegel2022high}, the ReLU activation function in~\cite{mao2023rates}, and the ReLU$^k$ activation function in~\cite{yang2024optimal}.

Approximation rates of deep neural networks have been derived for many function classes, such as continuous functions in~\cite{yarotsky2017error,yarotsky2018optimal,shen2019deep}, smooth functions in~\cite{yarotsky2020phase,lu2021deep}, piecewise‑smooth functions in~\cite{petersen2018optimal}, shift‑invariant spaces in~\cite{yang2022approximation}, and band‑limited functions in~\cite{montanelli2021deep}. Such approximation rates are usually characterized by the number of weights, as in~\cite{yarotsky2017error,yarotsky2018optimal,yarotsky2020phase}, or the number of neurons, as in~\cite{shen2019deep,lu2021deep}, rather than the size of the weights in some norms. Nevertheless, the size of the weights has an impact on the approximation capacity of neural networks. For example, regularization methods have been proposed to control the Lipschitz constant of neural networks to improve their robustness to adversarial attacks, see~\cite{miyato2018spectral,brock2019large,de2025thesis,guglielmi2025contractivity,de2025improving}. Such regularization methods often restrict certain norms of the weights, thereby substantially reducing the approximation capabilities of neural networks, see~\cite{de2025thesis,de2025approximation}. For instance,~\cite{huster2018limitations} showed that ReLU neural networks with certain constraints on the weights cannot approximate some simple functions, such as the absolute value function.

It is therefore natural to ask how norm constraints on the weights affect approximation capacity, and to characterize approximation rates in terms of the size of the weights. To the best of our knowledge, the first paper addressing this challenge is~\cite{JIAO2023249}. The authors study the approximation capacity of deep ReLU neural networks with norm constraints on the weights; in particular, they prove upper and lower bounds on the approximation error~\eqref{eq:app_err} of these networks for smooth function classes. Such upper and lower bounds are called \emph{approximation bounds}.

In this paper, we generalize the results of~\cite{JIAO2023249} to deep neural networks with arbitrary nonlinear, non-polynomial activation functions that are differentiable almost everywhere with derivatives bounded in $[0,1]$. In particular, we extend the derivation of the approximation upper and lower bounds in~\cite{JIAO2023249} to deep neural networks with arbitrary activation functions, and we provide a probabilistic version of the upper bound.

\subsection{Outline of the paper}

This paper is organized as follows. In~\cref{sec:preliminaries}, we introduce preliminary results on neural networks and define the Rademacher complexity. In~\cref{sec:upper_bound}, we establish the upper bound through a constructive approach, specifically, by first approximating high-degree monomials and then generalizing it to functions via a partition of unity and Taylor expansion. In~\cref{sec:lower_bound}, we establish the lower bound following a well-known strategy from \cite{JIAO2023249} that easily generalizes to our case thanks to the results obtained in the preliminaries regarding Rademacher complexity. In~\cref{sec:upper_bound_nd}, we construct neural networks with randomly sampled weights and biases to approximate functions, giving a probabilistic point of view of the results obtained in~\cref{sec:upper_bound}. Finally, in~\cref{subsec:lower_regularity,sec:numerical_weaker}, we establish that the assumption used throughout the paper regarding the regularity of the activation function $\sigma$ can be significantly weakened without worsening the approximation error.

\section{Preliminaries} \label{sec:preliminaries}

\subsection{Neural networks}

This section introduces the notation and preliminary results on neural networks used throughout the paper. We denote by $\nN_{d,n}(W,L)$ the set of neural networks with width $W$, depth $L$, and input and output dimensions $d$ and $n$ respectively. When the input and output dimensions are clear from context, we write simply $\nN(W,L)$. 

Sometimes, we will use the notation $\phi_\theta\in\nN(W,L)$ to emphasize that a neural network $\phi_\theta$ is parameterized by the weights and biases
\begin{equation}\label{eq:param}
    \theta := \bigl((A_0,b_0),(A_1,b_1),\ldots,(A_L,b_L)\bigr).
\end{equation}
If $\sigma_\ell(1)\ne 0$, $\ell=0,\dots,L-1$, the neural network described in \eqref{eq:nn1} can also be expressed compactly as
\begin{equation}\label{eq:nn2}
    \phi(x) = \Tilde{A}_L\sigma_{L-1} \left(\Tilde{A}_{L-1}\sigma_{L-2} \left(\cdots\Tilde{A}_1\sigma_0(\Tilde{A}_0\Tilde{x} ) \cdots \right) \right),
\end{equation}
where the augmented matrices $\Tilde A_0, \Tilde A_1,\ldots,\Tilde A_L$ are defined as
\[
    \Tilde{A}_0 = 
    \begin{pmatrix}
    A_0 & b_0 \\
    0 & 1
    \end{pmatrix},
    \quad
    \Tilde{A}_\ell = 
    \begin{pmatrix}
    A_\ell & (1/\sigma_{\ell-1}(1))b_\ell \\
    0 & 1/\sigma_{\ell-1}(1)
    \end{pmatrix},
    \quad
    \Tilde{A}_L = 
    \begin{pmatrix}
    A_L & (1/\sigma_{L-1}(1))b_L
    \end{pmatrix},
\]
while the augmented input vector is $\Tilde x = (x,1)^\top$. For simplicity, we assume throughout the paper $\sigma_\ell(1)=1$ for $\ell=0,\dots,L-1$, but our results remain valid for any value of $\sigma_\ell(1)$ different from $0$. 

A natural way to impose a norm constraint on the weights of a neural network is as follows: given $K > 0$ and matrix norm $\|\cdot\|$, we require the weights of $\phi_{\theta}$ to satisfy
\[
    \prod_{\ell=0}^L \|\Tilde{A}_\ell\| \le K.
\]
In this paper, we consider the uniform operator norm for a matrix $A=(a_{ij}) \in \R^{m \times n}$ defined as
\[
    \|A\| := \|A\|_\infty = \max_{\|x\|_\infty \le 1} \|Ax\|_\infty.
\]
We recall that $\|A\|$ corresponds to the maximum $1$-norm of the rows of $A$
\[
    \|A\|=\max_{1 \le i \le m} \sum_{j=1}^n |a_{ij}|.
\]
Thus, the product of the weight norms can be rewritten as follows
\[
    \prod_{\ell=0}^L \left\|\Tilde{A}_\ell \right\| = \left\|(A_L,b_L) \right\| \prod_{\ell=0}^{L-1} \max\left\{ \left\|(A_\ell,b_\ell) \right\|, 1 \right\}.
\]
We are now ready to define the space of norm constrained neural networks.

\begin{definition}
    We define the \emph{space of norm constrained neural networks with width $W$ and depth $L$}, denoted by $\nN(W,L,K)$, as the set of functions $\phi_\theta\in\nN(W,L)$, parameterized by weights and biases $\theta$ as in~\eqref{eq:param}, that satisfy the norm constraint
    \begin{align*}
        &\left\|(A_\ell,b_\ell) \right\| \le 1, \quad \text{for} \quad \ell\in\{0,\dots,L\}\setminus I,\\
        &\left\|(A_\ell,b_\ell) \right\| \ge 1, \quad \text{for} \quad \ell\in I, \quad \text{and} \quad \prod_{\ell\in I} \left\|(A_\ell,b_\ell) \right\| \le K,
    \end{align*}
    where $I\subset\{0,\dots,L\}$.

    In the specific case where $b_\ell = 0$, $\ell=0, \dots, L$, we define the space of bias-free norm constrained neural networks, and we denote it by $\sNN(W,L,K)$.
\end{definition}

\begin{remark}
    We notice that $\prod_{\ell=0}^L \left\|\Tilde{A}_\ell \right\|$ is the Lipschitz constant of the neural network $\phi_\theta\in\nN(W,L,K)$, so constraining the norm of the weights of a neural network means constraining its Lipschitz constant. In practice, we may therefore estimate $K$ by computing the Lipschitz constant of $\phi_\theta$, treating $K$ as an upper bound on this quantity.
\end{remark}

The next proposition generalizes~\cite[Proposition 2.1]{JIAO2023249} to any activation function.

\begin{proposition} \label{proposition:inclusion}
The following inclusions hold:
\[
    \sNN(W,L,K) \subset \nN(W,L,K) \subset \sNN(W+1,L,K).
\]
\end{proposition}

\begin{proof}
The first inclusion is straightforward, as the elements in $\sNN(W, L, K)$ represent neural networks that have zero biases.  The second inclusion can be established by using the representation of a neural network with non-zero biases in terms of the augmented matrices, as shown in~\eqref{eq:nn2}.
\end{proof}

The next proposition generalizes~\cite[Proposition 2.5]{JIAO2023249}, and provides an overview of fundamental operations that can be performed on neural networks. These operations will be particularly useful in constructing approximations of regular functions using neural networks.

\begin{proposition} \label{proposition:properties_neural_netowrk}
    Let $\phi_1\in\nN_{d_1,k_1}(W_1,L_1,K_1)$ and $\phi_2\in\nN_{d_2,k_2}(W_2,L_2,K_2)$. The following properties hold.
    
    \begin{enumerate}[label=(\roman*)]
         \item \textbf{Inclusion.} If $d_1 = d_2$, $k_1 = k_2$, $W_1 \leq W_2$, $L_1\leq L_2$, and $K_1 \le K_2$, then
         \[
            \nN_{d_1, k_1}(W_1, L_1, K_1) \subseteq \nN_{d_2, k_2}(W_2, L_2, K_2).
         \]
    
        \item \textbf{Composition.} If $k_1=d_2$, then the composition of $\phi_1$ and $\phi_2$ satisfies
        \[
        \phi_2\circ\phi_1\in\nN_{d_1,k_2} \left(\max\{W_1,W_2\},L_1+L_2,K_1 K_2 \right).
        \]
        Additionally, if $\phi_1(x)=Ax+b$, with $A \in \R^{d_2 \times d_1}$ and $b \in \R^{d_2}$, then the composition $\phi_2\circ\phi_1$ is a neural network that belongs to $\nN_{d_1,k_2} \left(W_2,L_2,\|A\|K_2 \right)$.
                
        \item \textbf{Concatenation.} If $d_1=d_2$, then
        \[
        \phi(x):=\bigl(\phi_1(x),\phi_2(x)\bigr) \in\nN_{d_1,k_1+k_2} \left(W_1+W_2,\max\{L_1,L_2\},\max\{K_1,K_2\} \right).
        \]
        
        \item \textbf{Linear Combination.} If $d_1=d_2$ and $k_1=k_2$, then, for any $c_1,c_2\in\R$, 
        \[
            \phi:=c_1\phi_1+c_2\phi_2 \in \nN_{d_1,k_1}\left(W_1+W_2,\max\{L_1,L_2\},|c_1|K_1+|c_2|K_2\right).
        \]
    \end{enumerate}
\end{proposition}

\begin{proof}
    We assume that the neural networks $\phi_1$ and $\phi_2$ are parameterized by weights and biases
    \[
        \theta^{(i)} = 
        \left(
        \begin{pmatrix}
            A_0^{(i)} & b_0^{(i)}
        \end{pmatrix},
        \begin{pmatrix}
            A_1^{(i)} & b_1^{(i)}
        \end{pmatrix},
        \ldots,
        \begin{pmatrix}
            A_{L_i}^{(i)} & b_{L_i}^{(i)}
        \end{pmatrix}
        \right), \qquad i = 1,2.
    \]

    \begin{enumerate}[label=(\roman*)]
        \item We consider $A_\ell^{(1)}\in\R^{W_2\times W_2}$ and $b_\ell^{(1)}\in\R^{W_2}$, with $0\le\ell\le L$, by adding suitable zero rows and columns to $A_\ell^{(1)}$ and $b_\ell^{(1)}$ if necessary. We notice that this operation does not change the norm. Then, by adding $L_2-L_1$ identity layers, we conclude that $\phi_1\in\nN_{d_2,k_2}(W_2,L_2,K_2)$.
                
        \item By ({\romannumeral 1}), without loss of generality, we assume $W_1=W_2$. Then $\phi_2\circ\phi_1$ can be parameterized by
        \[
            \theta =
            \left(
                \begin{pmatrix}
                    A_0^{(1)} & b_0^{(1)}
                \end{pmatrix},
                \ldots,
                \begin{pmatrix}
                    A_0^{(2)}A_{L_1}^{(1)} & A_0^{(2)}b_{L_1}^{(1)}+b_0^{(2)}
                \end{pmatrix},
                \ldots,
                \begin{pmatrix}
                    A_{L_2}^{(2)} & b_{L_2}^{(2)}
                \end{pmatrix}
            \right).
        \]
        Then, for any $x,y\in\R^{d_1}$, it holds that
        \[
            \|\phi_2\circ\phi_1(x)-\phi_2\circ\phi_1(y)\| \le K_1K_2 \,\|x-y\|.
        \]
        Hence $\phi_2\circ\phi_1\in\nN_{d_1,k_2}\left(\max\{W_1,W_2\},L_1+L_2, K_1K_2\right)$.

        In particular, if $\phi_1(x)=Ax+b$, we parameterize $\phi_2\circ\phi_1$ by
        \[
            \theta =
            \left(
                \begin{pmatrix}
                    A_0^{(2)}A & A_0^{(2)}b+b_0^{(2)}
                \end{pmatrix},
                \begin{pmatrix}
                    A_1^{(2)} & b_1^{(2)}
                \end{pmatrix},
                \ldots,
                \begin{pmatrix}
                    A_{L_2}^{(2)} & b_{L_2}^{(2)}
                \end{pmatrix}
            \right).
        \]
        Then, for any $x,y\in\R^{d_1}$, we have that
        \[
            \|\phi_2\circ\phi_1(x)-\phi_2\circ\phi_1(y)\| \le \|A\|K_2 \, \|x-y\|,
        \]
        which means that $\phi_2\circ \phi_1\in\nN_{d_1,k_2} \left(W_2,L_2,\|A\|K_2 \right)$.
                
        \item By ({\romannumeral 1}), without loss of generality, we assume $L_1=L_2=L$. We notice that $\phi$ can be parameterized by
        \[
            \theta = 
            \left(
            \begin{pmatrix}
                A_0 & b_0
            \end{pmatrix},
            \begin{pmatrix}
                A_1 & b_1
            \end{pmatrix},
            \ldots,
            \begin{pmatrix}
                A_L & b_L
            \end{pmatrix}
            \right),
        \]
        where we set
        \[  A_0 :=
            \begin{pmatrix}
                A_0^{(1)} \\
                A_0^{(2)}
            \end{pmatrix}, \quad
            b_0 :=
            \begin{pmatrix}
                b_0^{(1)} \\
                b_0^{(2)}
            \end{pmatrix}, \quad
            A_\ell :=
            \begin{pmatrix}
                A_\ell^{(1)} & 0 \\
                0 & A_\ell^{(2)}
            \end{pmatrix}, \quad
            b_\ell :=
            \begin{pmatrix}
                b_\ell^{(1)} \\[.3em]
                b_\ell^{(2)}
            \end{pmatrix},
            \]
        for $\ell = 1,\ldots,L$. Then, for any $x,y\in\R^{d_1}$, it holds that
        \begin{align*}
            \|\phi(x)-\phi(y)\| \le
            \left\|
                \begin{pmatrix}
                    \phi_1(x) \\
                    \phi_2(x)
                \end{pmatrix}
                -
                \begin{pmatrix}
                    \phi_1(y) \\
                    \phi_2(y)
                \end{pmatrix}
            \right\|
            & = \max\{ \|\phi_1(x)-\phi_1(y)\| , \|\phi_2(x)-\phi_2(y)\| \} 
            \\ & \le \max\{K_1,K_2\} \|x-y\|,
        \end{align*}
        which allows us to conclude that $\phi\in\nN_{d_1,k_1+k_2}\left(W_1+W_2,\max\{L_1,L_2\},\max\{K_1,K_2\}\right)$.
        
        \item By ({\romannumeral 1}), without loss of generality, we assume $L_1=L_2=L$. We notice that $\phi$ can be parameterized by
        \[
            \theta = 
            \left(
            \begin{pmatrix}
                A_0 & b_0
            \end{pmatrix},
            \begin{pmatrix}
                A_1 & b_1
            \end{pmatrix},
            \ldots,
            \begin{pmatrix}
                A_L & b_L
            \end{pmatrix}
            \right),
        \]
        where
        \begin{align*}
           & A_0 :=
            \begin{pmatrix}
                A_0^{(1)} \\
                A_0^{(2)}
            \end{pmatrix},\
            b_0 :=
            \begin{pmatrix}
                b_0^{(1)} \\
                b_0^{(2)}
            \end{pmatrix},\quad
            A_\ell :=
            \begin{pmatrix}
                A_\ell^{(1)} & 0 \\
                0 & A_\ell^{(2)}
            \end{pmatrix},\
            b_\ell :=
            \begin{pmatrix}
                b_\ell^{(1)} \\
                b_\ell^{(2)}
            \end{pmatrix}, \quad \ell=1,\ldots,L-1,
            \\ &
            A_L :=
            \begin{pmatrix}
                c_1A_L^{(1)} & c_2A_L^{(2)} \\
            \end{pmatrix},\
            b_L := c_1b_L^{(1)}+c_2b_L^{(2)}.
        \end{align*}
        Then, for any $x,y\in\R^{d_1}$, we have that
        \[
            \|c_1\phi_1(x)+c_2\phi_2(x) - c_1\phi_1(y)-c_2\phi_2(y)\| \le (|c_1|K_1+|c_2|K_2) \|x-y\|,
        \]
        which means that $c_1 \phi_1 + c_2 \phi_2 \in \nN_{d_1,k_1}\left(W_1+W_2,\max\{L_1,L_2\}, |c_1| K_1 + |c_2|K_2\right)$.
	\end{enumerate}
\end{proof}

\subsection{Rademacher complexity}

We now introduce the Rademacher complexity, a standard tool for measuring the richness of a function class with respect to a probability distribution. This notion will be central to establishing lower bounds on the approximation error for norm constrained neural networks.

\begin{definition}\label{def:rad}
    Let $S \subseteq \R^d$. The \emph{Rademacher complexity} of $S$ is defined by
    \[
    \mathcal R_d(S) = \mathbb E_{\boldsymbol{\xi}} \left[\sup_{s \in S} \frac1d \left|\sum_{i=1}^d \xi_i s_i\right| \right],
    \]
    where $s=(s_1,\ldots,s_d) \in \R^d$ and $\boldsymbol{\xi} = \{\xi_i\}_{i=1,\ldots,d}$ is a sequence of i.i.d. Rademacher random variables (i.e. taking the values $-1$ and $1$ with equal probability).
\end{definition}

We now generalize~\cite[Lemma 2.3]{JIAO2023249} to get upper and lower bounds to the Rademacher complexity of samples of neural networks with an arbitrary activation function $\sigma$.

\begin{lemma}\label{lemma:rad_compl}
For any $x_1,\dots, x_n\in[-B,B]^d$ with $B \ge 1$, we define
\[
S := \left\{\left(\phi( x_1),\ldots,\phi( x_n)\right)\ :\ \phi\in\sNN_{d,1}(W,L,K)\right\} \subseteq \R^n, \qquad s := \max_{1\le j\le d}\left(\frac1n\sum_{i=1}^n x_{i,j}^2\right)^{1/2},
\]
where $x_{i,j}$ denotes the $j$-th component of $x_i$. Then the following statements hold.
\begin{enumerate}[label=(\roman*)]
\item \textbf{Upper bound.} If $\sigma$ is 1-Lipschitz, then
\[
\mathcal R_n(S) \le \frac{BK}{\sqrt n} \sqrt{2\left(L+1+\log(d)\right)}.
\]
\item \textbf{Lower bound for ReLU/LeakyReLU activation function.} If $\sigma(z)=\max\{z,0\}$ or $\sigma(z)=z^+ + \alpha z^-$ with $\alpha\in[0,1)$, $z^+=\max\{z,0\}$ and $z^-=-\min\{z,0\}$, and $W,L\in\N$ with $W\ge 2$ and $L \geq 1$, then
\[
\mathcal R_n(S) \ge \frac{1}{2\sqrt2} (1-\alpha) K \frac{s}{\sqrt n}.
\]      
\item \textbf{Lower bound for a general activation function.} If $W\in\N$, $W\ge2$, $\sigma\in C^2([-\delta,\delta])$, with $\sigma(0)=0$, $\sigma'(0)>0$ and  $\max_{|z|\le\delta}|\sigma''(z)|\le M$, then there exists $n_0=n_0(\sigma,B)$ such that, for all $n\ge n_0$,
\[
\mathcal R_n(S) \ge c_*(\sigma,B) \, K\,\frac{s^2}{n},
\qquad \text{with} \quad
c_*(\sigma,B):=\frac{\sigma'(0)}{8 \,M B^2}.
\]
\end{enumerate}
\end{lemma}

\begin{proof}
We prove separately the three cases. \leavevmode

\begin{enumerate}[label=(\roman*)]
    \item By~\cite[Theorem 2]{pmlr-v75-golowich18a}, it holds that
    \[
    \mathcal R_n(S) \le \frac{BK}{\sqrt n} \sqrt{2\left(L+1+\log(d)\right)}.
    \]

    \item We show that the set $\sNN_{d,1}(W,L,K)$ contains a set of linear functions, providing a lower bound for the Rademacher complexity of $S$. For each coordinate $j\in\{1,\dots,d\}$, denoted by $e_j\in\R^d$ the $j$-th vector of the canonical basis, we construct a bias-free neural network $\phi_j$ on $\R^d$ of width 2, depth 1, and norm constraint $K$ as
    \[
    \phi_j(x) = A_1\sigma(A_{0,j}x), \quad x\in\R^d, \quad \text{where} \quad
    A_{0,j} = \begin{pmatrix}
        e_j^\top \\
        -e_j^\top
    \end{pmatrix}\in\R^{2\times d}
    \quad
    \text{and}
    \quad
    A_1 = \frac{K}{2}\begin{pmatrix}
        1 & -1
    \end{pmatrix}.
    \]
    We notice that $\|A_{0,j}\|=1$ and $\|A_1\|=K$, thus $\|A_{0,j}\|\|A_1\|=K$, and the function realized by the neural network is given by
    \[
    \phi_j(x)=\frac{K}{2} \left(\sigma( x_j)-\sigma(- x_j)\right)=\frac{K}{2}(1-\alpha) \bigl( x_j^+ - x_j^- \bigr)=\frac{K}{2}(1-\alpha) x_j.
    \]
Then, since $\sNN_{d,1}(2,1,K)\subset\sNN_{d,1}(W,L,K)$, we have that the space of bias-free neural networks contains the linear space $\{x \mapsto (1-\alpha)(K/2) x_j: j=1,\dots,d\}$, and we get
\begin{align*}
\mathcal R_n(S) &= \mathbb E_{\boldsymbol{\xi}} \left[\max_{\phi \in \sNN_{d,1}(W,L,K)} \frac1n \left|\sum_{i=1}^n \xi_i \phi( x_i)\right| \right] \ge \frac{(1-\alpha)K}{2} \frac1n \E_{\boldsymbol\xi}\left[\max_{1\le j\le d}\left|\sum_{i=1}^n \xi_i x_{i,j}\right|\right] \\
&\ge \frac{(1-\alpha)K}{2} \frac{1}{\sqrt2 \, n}\, \max_{1\le j\le d}\left(\sum_{i=1}^n x_{i,j}^2\right)^{1/2},
\end{align*}
by using $\E[\max_j Y_j]\ge \max_j \E[Y_j]$ and \emph{Khintchine inequality} $\E|\sum_i \xi_i u_i|\ge (1/\sqrt{2})\|u\|_2$ (\cite{haagerup1981best} and~\cite{ledoux1991probability}[Lemma 4.1]). Since $\sum_{i}  x_{i,j}^2 = n\,\left(\frac1n\sum_i x_{i,j}^2\right)$, we get the desired lower bound.

\item Similarly to the previous point, we show that the set of bias-free neural networks contains a set of functions that are linear up to an error term, providing a lower bound for the Rademacher complexity of $S$. 

We start by defining the remainder function
\[
    r(z):=\sigma(z)-\sigma'(0)\,z, \quad \text{with} \quad |r(z)|\le Mz^2\quad\text{for }|z|\le\delta, \quad M,\delta>0,
\]
and, for each coordinate $j\in\{1,\dots,d\}$ and $\eps>0$ to be constrained later, we construct a bias-free neural network of width 2, depth 1, and norm constraint $K$ as
\[
    \phi_{j,\eps}(x) = A_1\sigma(A_{0,j,\eps}x), \quad x\in\R^d, \quad \text{where} \quad
    A_{0,j,\eps} = \begin{pmatrix}
        \eps e_j^\top \\
        0
    \end{pmatrix}\in\R^{2\times d}
    \quad
    \text{and}
    \quad
    A_1 = \frac{K}{\sigma'(0)}\begin{pmatrix}
        1 & 0
    \end{pmatrix}.
\]
Here $\|A_{0,j,\eps}\|=\eps$ and $\|A_1\|=K/\sigma'(0)$, thus $\|A_{0,j,\eps}\|\|A_1\|=\eps K / \sigma'(0)$, which is bounded from above by $K$ if $\eps\le\sigma'(0)$. Furthermore, we choose $\eps<\frac{\delta}{B}$ to satisfy the condition that, for each $x\in[-B,B]^d$, $|\eps x_j|\le\delta$ for each coordinate $j$. Thus, the function realized by the neural network
\[
    \phi_j(x)=\frac{K}{\sigma'(0)}\sigma(\eps x_j)=K \eps x_j + \frac{K}{\sigma'(0)}r(\eps x_j) =: K\eps x_j + \Delta_{j,\eps}(x_j),
\]
is linear up to an error term $\Delta_{j,\eps}(x_j)$ bounded from above as follows:
\[
    \left| \Delta_{j,\eps}(x_j) \right| \le \frac{K}{\sigma'(0)} M \eps^2 x_j^2 \le \frac{K}{\sigma'(0)} M \eps^2 B^2.
\]
If we define $\mathcal F_\eps:=\{\phi_{j,\eps}\ : \ 1\le j\le d\}$, since $\sNN_{d,1}(2,1,K)\subset\sNN_{d,1}(W,L,K)$, we have shown that the space of bias-free neural networks $\sNN_{d,1}(W,L,K)$ contains the space $\mathcal F_\eps$. By using the reverse triangle inequality, the fact that $\E[\max_j Y_j]\ge \max_j \E[Y_j]$, the Khintchine inequality $\E|\sum_i \xi_i u_i|\ge (1/\sqrt{2})\|u\|_2$ (\cite{haagerup1981best} and~\cite{ledoux1991probability}[Lemma 4.1]), and $\sum_{i} x_{i,j}^2 = n\,\left(\frac1n\sum_i x_{i,j}^2\right)$, we obtain that
\begin{align*}
\mathcal R_n(S) \ge \mathcal R_n(\mathcal F_\eps) &= \frac1n \E_{\boldsymbol\xi}\left[\max_{1\le j\le d}\left|\sum_{i=1}^n \xi_i (K \eps x_{i,j} + \Delta_{j,\eps}( x_i)\right|\right] \\
&\ge \frac{1}{n}\E_{\xi}\left[\max_{1\le j\le d}\left|\sum_{i=1}^n \xi_i\,K \eps\ x_{i,j}\right|\right] -\frac{1}{n}\sum_{i=1}^n\max_{1\le j\le d} |\Delta_{j,\eps}( x_i)| \\
&\ge K \eps\,\frac{1}{\sqrt2}\,\frac{s}{\sqrt n}-\frac{K}{\sigma'(0)} \, M\,\eps^2B^2.
\end{align*}
Therefore
\begin{equation}\label{eq:lb}
    \mathcal R_n(\mathcal F_\eps) \ge K \left( \frac{s}{\sqrt {2n}}\eps-\frac{1}{\sigma'(0)} MB^2\eps^2\right) =: K(\alpha_1\eps - \alpha_2\eps^2),
\end{equation}
and the right-hand side is maximized at
\[
    \eps^\star=\frac{\alpha_1}{2\alpha_2} =\frac{\sigma'(0)}{2\sqrt{2} M B^2}\cdot\frac{s}{\sqrt n}.
\]
We require
\[
    \eps^\star \le \min \left\{ \frac{\delta}{B} , \sigma'(0) \right\},
\]
which holds for all $n\ge n_0(\sigma,B)$. Substituting $\eps^\star$ into \eqref{eq:lb} yields
\[
    \mathcal R_n(S)\ \ge\ \mathcal R_n(\mathcal F_{\eps^\star}) \ge  K\cdot \frac{\alpha_1^2}{4\alpha_2} = K\frac{\sigma'(0)}{8 \, M B^2}\cdot \frac{s^2}{n} = \frac{\sigma'(0)}{8 \, M B^2} K \frac{s^2}{n},
\]
which concludes the proof.
\end{enumerate}
\end{proof}

\begin{corollary}
     The upper and lower bounds in~\cref{lemma:rad_compl} remain valid for neural networks in $\nN(W,L,K)$.
\end{corollary}

This corollary follows immediately thanks to~\cref{proposition:inclusion}. Before moving on, we need the following definition and notations.

\begin{definition}
    Let $\rho$ be a metric on $\R^d$ and $S\subset\R^d$. For $\eps>0$, a set $U\subset S$ is called an \textbf{$\eps$-packing} of $S$ if any two distinct elements $x,y\in U$ satisfy $\rho(x,y)>\eps$. Furthermore, the $\eps$-packing number of $S$ is defined as
    \[
        \mathcal{N}_p(S,\rho,\eps) := \max\bigl\{|U| \ : \ U \text{ is an $\eps$-packing of $S$} \bigr\}. 
    \]
\end{definition}

\noindent
\textbf{Notations.} For a multi-index $\mathbf s \in \N^d$, we denote its 1-norm as
    \[
        \| \mathbf s \|_1 := |s_1| + \cdots + |s_d|.
    \]
    Similarly, we denote the partial derivative of a function $f$ with respect to a multi-index $\mathbf s \in \N^d$ as
    \[
        \partial^{\mathbf s} f := \frac{\partial^{s_1}}{\partial x_1^{s_1}} \cdots \frac{\partial^{s_d}}{\partial x_d^{s_d}} f.
    \]

\subsection{Main results} \label{subsec:main_results}

In this section, we state our main results. First, we define the class of smooth functions that we are approximating with norm constrained neural networks.

\begin{definition}[Lipschitz-regular functions]\label{def:lipschiotz-regular}
    Let $d \in \N$ and $r = m+ \beta >0$, where $m \in \N$ and $\beta \in (0,1]$. The class of Lipschitz-regular functions in $\R^d$ is defined as
    \[
        \Lip_{r}(\R^d) := \left\{ f : \R^d \rightarrow \R \ : \ \max_{ \|\mathbf s\|_1 \le m} \sup_{x\in \R^d} |\partial^{\mathbf s} f(x) | \le 1, \ \max_{\|\mathbf s\|_1 =m} \sup_{x \neq y} \frac{|\partial^{\mathbf s} f(x) - \partial^{\mathbf s} f(y)|}{\|x-y\|_\infty^\beta} \le 1 \right\}.
    \]
    We denote by $\Lip_r$ the restriction of this class to the unit cube $[0,1]^d$. Furthermore, if $f\in\Lip_{r}(\R^d)$, then we define its $\Lip_{r}(\R^d)$-norm as
    \[
        \|f\|_{\Lip_{r}(\R^d)} := \max_{ \|\mathbf s\|_1 \le m} \sup_{x\in \R^d} |\partial^{\mathbf s} f(x) | + \max_{\|\mathbf s\|_1 =m} \sup_{x \neq y} \frac{|\partial^{\mathbf s} f(x) - \partial^{\mathbf s} f(y)|}{\|x-y\|_\infty^\beta}.
    \]
    An equivalent norm would be the one with the sum ($\sum$) in place of the $\max$.
\end{definition}

The class $\Lip_{r}(\R^d)$ is the set of functions $m$ times continuously differentiable, whose derivatives are bounded, and their $m$-th derivatives are $\beta$-Hölder continuous (for example, $\Lip_{1}(\R^d)$ is the set of bounded 1-Lipschitz functions).

Then, we introduce our fundamental assumption on the activation function $\sigma$.

\begin{assumption} \label{ass:sigma_exp}
    We suppose that $\sigma$ satisfies
    \begin{equation*}
        \sigma(t)= a_1 t+a_2 t^2+a_3 t^3+r(t), \quad \text{with} \quad |r(t)|\le M |t|^4 \quad \text{for all } |t|\le \rho,
    \end{equation*}
    for some $0<\rho\le1$, $M>0$ and coefficients $a_1,a_2\neq 0,a_3\in\mathbb R$.
\end{assumption}

Our objective is to study how well norm constrained neural networks in $\nN(W,L,K)$ approximate smooth functions in $\Lip_r$. In particular, we estimate the maximum error performed when approximating any function $f \in \Lip_r$ by a neural network $\phi \in \nN(W,L,K)$, called approximation error, defined as follows
\[
    \mathcal E \left(\Lip_r, \nN(W,L,K) \right) := \sup_{f \in \Lip_r} \inf_{\phi \in \nN(W,L,K)} \| f - \phi \|_{\infty},
\]
where $\|\cdot\|_\infty$ is the sup-norm on continuous functions over $[0,1]^d$, i.e.
\[
	\| f - \phi \|_{\infty} := \max_{x \in [0,1]^d} \left| f(x) - \phi(x) \right|.
\]
Within this framework, the main results of this paper can be summarized in the following theorems.

\begin{theorem}\label{thm:main}
    We fix $d \in \N$ and $r =m + \beta$, where $m \in \N$ and $\beta \in (0,1]$. The following bounds hold.
    \begin{enumerate}[label=(\roman*)]
        \item \emph{\textbf{Upper bound.}} If Assumption \ref{ass:sigma_exp} holds, then there exist constants $c_1,c_2,C_1,C_2>0$ such that, for all $W\in\N$ with $W\ge c_1$, $K\in\R$ with $K\ge c_2W$, and $L\in\N$ with $L \ge 2\big\lceil\log_2 d\big\rceil$, it holds that
        \begin{equation}\label{eq:main_upper_smooth}
            \mathcal E\left(\Lip_r,\,\nN(W,L,K)\right) \le C_1\,W^{-r/d} + C_2\,K^{-1}.
        \end{equation}
        
       	\item \emph{\textbf{Lower bound.}} If $d > 2r$, $W,L \in \N$ with $W \ge 2$ and $K \ge 1$, then we have that
        \[
            \mathcal E (\Lip_r,\nN(W,L,K)) \gtrsim \left( K^2 L \right)^{-\frac{r}{d-2r}}.
        \]
	\end{enumerate}
\end{theorem}

Finally, we present an alternative version of the approximation upper bound, where the approximation error is measured in probability. If $f \in \Lip_r$, $\eps>0$ and $\delta\in(0,1)$, we say that a neural network $\phi \in \nN(W,L,K)$ approximates $f$ on $[0,1]^d$ up to error $\eps$ with probability at most $1-\delta$ if
\[
    \mathbb{P}(\{x\in[0,1]^d  \ : \ |f(x)-\phi(x)|\le\eps\})\ge1-\delta.
\]
For simplicity, we write $\mathbb{P}(\{x\in[0,1]^d  \ : \ |f(x)-\phi(x)|\le\eps\})$ as $\mathbb{P}( |f(x)-\phi(x)|\le\eps)$.

\begin{theorem}\label{thm:main_pr}
We fix $d \in \N$ and $r =m + \beta$, where $m \in \N$ and $\beta \in (0,1]$, and we suppose that the activation function $\sigma$ satisfies Assumption \ref{ass:sigma_exp}. Then, there exist constants $c_1,c_2,C_1,C_2,C_3>0$ such that, for all $W\in\N$ with $W\ge c_1$, $K\in\R$ with $K\ge c_2W$, and $L\in\N$ with $L \ge 2\big\lceil\log_2 d\big\rceil$, it holds that, for any $\eps>0$,
\begin{equation}\label{eq:main_upper_smooth_pr}
   \mathbb{P}\Biggl(\mathcal E\left(\Lip_r,\,\nN(W,L,K)\right) \le  \underbrace{C_1\left(W^{-r/d}+ K^{-1}\right)}_{\text{deterministic error}} + \underbrace{ \eps }_{\text{probabilistic error}}\Biggr) \geq \biggl( 1 - C_2 \exp\biggl( -C_3 K \eps^2  \biggr) \biggr).
\end{equation} 

\end{theorem}

\begin{remark}
    We notice that the approximation error decomposes into a deterministic term (from approximating monomials by neural networks) and a probabilistic term (from randomized neural network weights). Furthermore, the interplay between the smallness of the deterministic error and the growth of the network width $W$ and weight norms $K$ persists: as $K \to \infty$, the probabilistic error vanishes almost surely, while the deterministic error can be controlled also by increasing $W$.
\end{remark}

\section{Upper bound in the deterministic case} \label{sec:upper_bound}

We establish in this section the upper bound in~\cref{thm:main} through a constructive approach. Our strategy approximates functions $f\in\Lip_r$ by approximating their local Taylor polynomials with norm constrained neural networks and by gluing those approximations together with a partition of unity.

By following the methodology developed in~\cite{yarotsky2017error,yarotsky2018optimal,yarotsky2020phase,lu2021deep,JIAO2023249}, we begin with the one-dimensional approximation of $f(x)=x^2$ on $[0,1]$. This explicitly reveals two key aspects of our construction: the cancellation mechanisms that enable efficient approximation, and the precise scaling of network parameters with respect to the approximation error.

Having understood these fundamental principles, we then extend our construction to high degree polynomials in arbitrary dimensions. Since the Taylor expansion is linear, this naturally generalizes to arbitrary functions in $\Lip_r$.

\subsection{Quadratic approximation}

We approximate in this section $f(x) = x^2$ on $[0,1]$ with norm constrained neural networks. Our analysis yields explicit bounds that reveal the interplay between network architecture (width and depth), constraints on the weight norm $K$, and approximation quality. Specifically, we demonstrate how these architectural choices jointly determine both the convergence rate and the resulting uniform approximation error.

\begin{lemma}\label{lemma:approximation_x2_regular}
Let $\alpha>0$ and let $\sigma:\R\rightarrow\R$ be an activation function satisfying Assumption \ref{ass:sigma_exp}. For any $k\in\N$, $k\ge k_0:=\max\{1,\rho^{-2/\alpha}\}$, there exists a neural network $\phi_{k,\alpha}\in\nN(2,2,K_{k,\alpha})$, with
\begin{equation} \label{def:K for x^2}
K_{k, \alpha} := \frac{k^\alpha}{|a_2|},
\end{equation}
such that $\phi_{k,\alpha}([0,1])\subseteq [0,1]$ and $\phi_{k,\alpha}(x)$ approximates the quadratic function $f(x) = x^2$ with the following error
\begin{equation} \label{eq:approximation_x2}
    \max_{x \in [0,1]} \left|x^2-\phi_{k,\alpha}(x)\right| \le \frac{M}{|a_2|}\,k^{-\alpha}.
\end{equation}
\end{lemma}

\begin{proof}
We define the width-$2$ neural network
\[
\Phi_{k,\alpha}(x) := d_k \, \Big(\sigma(wx)+\sigma(-wx)\Big).
\]
We set $w:=k^{-\alpha/2}$. For $k\ge k_0$, we have that $|wx|\le w\le \rho$ for all $x\in[0,1]$, so both $\sigma(wx)$ and $\sigma(-wx)$ can be expanded within this radius. Odd terms cancel by symmetry, hence we have that
\[
 \Phi_{k,\alpha}(x) = d_k \, \Big( 2a_2 w^2 x^2 + r(wx)+r(-wx) \Big) .
\]
By choosing $d_k = \frac{1}{2a_2 w^2}$, we ensure that the leading term of the expansion has a unit coefficient
\[
  \Phi_{k,\alpha}(x) = x^2 + \frac{k^\alpha}{2a_2}\left(r\bigl(k^{-\alpha/2}x\bigr)+r\bigl(-k^{-\alpha/2}x\bigr)\right).
\]
The error term in this expansion is bounded
\[
  \bigl|x^2-\Phi_{k,\alpha}(x)\bigr| \le \frac{k^\alpha}{2|a_2|} \,\left( \bigl|r\bigl(k^{-\alpha/2}x\bigr) \bigr| + \bigl| r\bigl(-k^{-\alpha/2}x\bigr) \bigr|\right) \le \frac{M}{|a_2|}\,k^{-\alpha},
\]
since $|r(\pm wx)|\le M|wx|^4\le M w^4$, proving the approximation error bound. Moreover, $\Phi_{k,\alpha}\in \nN \left(2,1,K_{k, \alpha}\right)$ with the norm constraint $K_{k, \alpha}$ as in~\eqref{def:K for x^2} due to
$k^{-\alpha/2} \le 1$. Width and depth of $\Phi_{k,\alpha}$ follow from the explicit construction. We note however that $\Phi_{k,\alpha}$ does not satisfy the range constraint. Indeed,
\[
    \Phi_{k,\alpha}\bigl([0,1] \bigr)\subset \bigl[-C_0 k^{-\alpha},\,1+C_0 k^{-\alpha}\bigr], \qquad \text{where } C_0:=M/|a_2|.
\]
We define $\psi(z):=\mathrm{ReLU}(z)-\mathrm{ReLU}(z-1)\in\nN(2,1,1)$ and we set $\phi_{k,\alpha}:=\psi\circ \Phi_{k,\alpha}$. Then $\phi_{k,\alpha}([0,1])\subseteq [0,1]$ and
\[
  \bigl|x^2-\phi_{k,\alpha}(x)\bigr| = \bigl|\psi\left(\Phi_{k,\alpha}(x)\right)-\psi\left(x^2\right)\bigr| \le \bigl|\Phi_{k,\alpha}(x)-x^2\bigr| \le \frac{M}{|a_2|}\,k^{-\alpha}, \qquad \text{for all } x \in [0,1],
\]
since $\psi$ is $1$-Lipschitz and $\psi(z)=z$ for all $z\in[0,1]$. In particular, $\phi_{k,\alpha}\in \nN \left(2,2,K_{k, \alpha}\right)$ as a consequence of the composition rule in~\cref{proposition:properties_neural_netowrk}. The error bound~\eqref{eq:approximation_x2} follows.
\end{proof}

    

\begin{remark}[Optimality and extensions]
Two natural questions arise from~\cref{lemma:approximation_x2_regular}.
\begin{enumerate}
    \item[\textnormal{a)}] Can exact clipping be achieved without ReLU activations?
    \item[\textnormal{b)}] Is the error norm trade-off optimal?
\end{enumerate}
We answer both in~\cref{appendix:additional}:~\cref{prop:no_clipping} shows that exact clipping requires piecewise linear activations, highlighting the role of ReLU-based functions.~\cref{prop:opt_scaling_even2} establishes the optimality of~\eqref{eq:approximation_x2}, confirming that no significant improvement in the error norm balance is possible within the norm constrained setting.
\end{remark}

\subsection{Approximation of multivariate monomials} 

We demonstrate in this section that monomials of the form $f(x_1, \dots, x_d) = x_1 x_2 \cdots x_d$ can be approximated by neural networks with an error of $\mathcal{O}(k^{-\alpha})$. We start from the approximation of $f(x_1, x_2) = x_1x_2$ by using the result of~\cref{lemma:approximation_x2_regular}. The networks are designed with appropriate weights, depths, and norms.

\begin{lemma}\label{lemma:approx_xy}
Let $\alpha>0$ and let $\sigma:\R\rightarrow\R$ be an activation function satisfying Assumption \ref{ass:sigma_exp}. For any $k\in\N$, $k\ge k_0:=\max\{1,(2/\rho)^{2/\alpha}\}$, there exists a neural network $\phi_{k,\alpha} \in \nN \bigl(6,2,K_{k,\alpha}\bigr)$, with
\[
K_{k,\alpha}=\tfrac32 \tfrac{k^\alpha}{|a_2|},
\]
such that $\phi_{k,\alpha}\bigl( [-1,1]^2 \bigr) \subseteq [-1,1]$ and $\phi_{k,\alpha}(x,y)$ approximates the function $f(x,y) = xy$ with following error estimate
\begin{equation} \label{lem3:err_estim}
\max_{(x,y) \in [-1,1]^2} \, \bigl|xy-\phi_{k,\alpha}(x,y)\bigr| \le 9\frac{M}{|a_2|}\, k^{-\alpha}.
\end{equation}
\end{lemma}

\begin{proof}
We set $w:=k^{-\alpha/2}$. For $|t|\le 2$, we have that $|wt|\le 2w\le \rho$ by the choice of $k$, hence
\[
\sigma(wt)+\sigma(-wt)=2a_2 (wt)^2 + \left(r(wt)+r(-wt)\right),\qquad \text{with} \quad |r(\pm wt)|\le M|wt|^4.
\]
We set $d_k := \tfrac{1}{2a_2 w^2}=\tfrac{k^\alpha}{2a_2}$ and we define the two-neuron block
\[
S_{k,\alpha} (t):=d_k \, \bigl( \sigma(wt)+\sigma(-wt)\bigr) = \frac{k^\alpha}{2a_2} \bigl(\sigma\bigl(k^{-\alpha/2}\,t\bigr) + \sigma\bigl(-k^{-\alpha/2}\,t\bigr)\bigr) ,
\]
and the bilinear approximation
\[ 
\Phi_{k,\alpha}(x,y) := \tfrac12\left(S_{k,\alpha}(x + y)-S_{k,\alpha}(x)-S_{k,\alpha}(y)\right).
\]
Then we have that the square error is
\[
E(t) := S_{k,\alpha}(t) - t^2 = \frac{r\bigl(k^{-\alpha/2}\,t\bigr) + r\bigl(-k^{-\alpha/2}\,t\bigr)}{2a_2 k^{-\alpha}} \implies |E(t)|\le \frac{M}{|a_2|}k^{-\alpha} \, t^4.
\]
For $t\in[-1,1]$, this yields $|E(t)|\le \tfrac{M}{|a_2|}k^{-\alpha}$, and for $t\in[-2,2]$, we have that $|E(t)|\le 16\,\tfrac{M}{|a_2|}k^{-\alpha}$. Since
\[
\Phi_{k,\alpha}(x,y)-xy=\Phi_{k,\alpha}(x,y)-2\left(\left(\frac{x+y}{2}\right)^2-\left(\frac{x}{2}\right)^2-\left(\frac{y}{2}\right)^2\right)=\tfrac12\Bigl(E(x + y) - E(x) - E(y)\Bigr),
\]
for $(x,y)\in[-1,1]^2$, we have the following error estimate:
\[
\bigl| \Phi_{k,\alpha}(x,y)-xy \bigr| \le \frac12\left(16\frac{M}{|a_2|}k^{-\alpha} +2 \frac{M}{|a_2|}k^{-\alpha} \right) = 9\frac{M}{|a_2|}\,k^{-\alpha}.
\]
We notice that $\Phi_{k,\alpha}$ is a six-neuron block with input weights
\[
\pm (w,\ w),\quad \pm (w,\ 0),\quad \pm (0,\ w), \qquad \text{and output weights} \quad \left(\tfrac{d_k}{2},\tfrac{d_k}{2},-\tfrac{d_k}{2},-\tfrac{d_k}{2},-\tfrac{d_k}{2},-\tfrac{d_k}{2}\right). 
\]
Thus $\Phi_{k,\alpha}\in\nN(6,1,K_{k,\alpha})$, with $K_{k,\alpha}=\max\{2|w|,1\}(6 \cdot |d_k|/2)$; since $2k^{-\frac{\alpha}{2}} \le \rho \le 1$, we have that 
\[
    K_{k,\alpha} = \tfrac32 \tfrac{k^\alpha}{|a_2|}.
\]
We now consider $\psi(z):=\mathrm{ReLU}(z+1)-\mathrm{ReLU}(z-1)-1\in\nN(2,1,1)$ and we set $\phi_{k,\alpha} := \psi\circ \Phi_{k,\alpha}\in\nN(6,2,K_{k,\alpha})$. Then, the conclusion follows from the $1$-Lipschitz property of $\psi$ and $\psi(z)=z$ on $[-1,1]$ as in~\cref{lemma:approximation_x2_regular}. 
\end{proof}

To generalize this result to arbitrary monomials of the form $f(x_1, \dots, x_d) = x_1x_2 \cdots x_d$, we use the neural network obtained in~\cref{lemma:approx_xy} repeatedly approximating non overlapping pairs of consecutive monomials $x_i  x_{i+1}$.

\begin{lemma}\label{lemma:approx_xy_d}
Let $\alpha>0$ and let $\sigma:\R\rightarrow\R$ be an activation function satisfying Assumption \ref{ass:sigma_exp}. For any $k\in\N$, $k\ge k_0:=\max\{1,(2/\rho)^{2/\alpha}\}$, there exists a neural network $\phi_{k,\alpha}^{(d)}\in\nN\left(W,L,K\right)$, with
\begin{equation*} \label{eq:lem4:params}
W := W_d = 6\Big\lceil\frac{d}{2}\Big\rceil,\qquad
L := L_d = 2\Big\lceil\log_2 d\Big\rceil,\qquad
K := K_{k,\alpha} = \left[\frac{3}{2}\frac{k^\alpha}{|a_2|}\right]^{\lceil\log_2 d\rceil},
\end{equation*}
such that  $\phi_{k,\alpha}^{(d)} \bigl( [-1,1]^d \bigr) \subseteq [-1,1]$ and $\phi_{k,\alpha}^{(d)}(x_1, \dots, x_d)$ approximates the function $f(x_1, \dots, x_d) = \prod_{i=1}^d x_i$ with following error estimate
\[
\max_{x\in[-1,1]^d}\Bigl|x_1 \cdot \dots \cdot x_d -\phi_{k,\alpha}^{(d)}(x_1, \dots, x_d)\Bigr| \le \left(2^{\lceil\log_2 d\rceil}-1\right)C_{\star}\,k^{-\alpha},
\]
where $C_\star:=9 M/|a_2|$ is the constant from~\cref{lemma:approx_xy}.
\end{lemma}

\begin{proof}
Let the function $\phi^{(2)}_{k,\alpha}(x_1, x_2)$ be constructed as in~\cref{lemma:approx_xy}. We define $\phi_{k,\alpha}^{(d)}(x_1, \dots, x_d)$ as a result of the following $D:=\lceil\log_2 d\rceil$ steps recursive procedure. At zero step, we introduce a finite sequence of functions
\[
z^{(0)}_j:=x_j, \qquad j=1,\ldots, d.
\]
At each step $1 \leq \ell \leq D$, we multiply the elements of the sequence pair wisely: the first by the second one, the third by the fourth one, and so on. Thus, we construct the next level functions
\[
    y^{(\ell)}_j := \phi_{k,\alpha}^{(2)}\left(z^{(\ell-1)}_{2j-1},z^{(\ell-1)}_{2j}\right),\qquad j=1,\dots,\Big\lceil\tfrac{d}{2^\ell}\Big\rceil.
\]
If $\lceil d/2^\ell\rceil$ is odd, we leave the last component unchanged, i.e. $y^{(\ell)}_{\lceil d/2^\ell\rceil}:=z^{(\ell-1)}_{2\lceil d/2^\ell\rceil-1}$. In the end of each step we set $z^{(\ell)}:=y^{(\ell)}$, and we define the output of the neural network $\phi_{k,\alpha}^{(d)}(x):= z^{(D)}_1$.

The recursive procedure described above is a neural network. Indeed, for each step $\ell=1,\dots,D$ and $j=1,\dots,\lceil d/2^\ell\rceil$, we introduce the projection operators 
\[
\operatorname{Proj}_j^\ell \left(z^{(\ell-1)}_1, \ldots, z^{(\ell-1)}_{
\big\lceil \tfrac{d}{2^\ell} \big\rceil} \right) := \left( z^{(\ell-1)}_{2j-1} , z^{(\ell-1)}_{2j} \right)
\] 
that define the input of each $\phi^{(2)}_{k,\alpha}(\cdot, \cdot)$ from the previous step. By introducing $\phi_{k,\alpha,\ell,j}^{(2)}:=\phi_{k,\alpha}^{(2)}\circ\operatorname{Proj}_j^\ell$, we realize 
\[
    y^{(\ell)}_j =  \phi_{k,\alpha,\ell,j}^{(2)}\left(z^{(\ell-1)}\right),
\]
and thus the output of each layer $\ell$ is defined as
\[
    y^{(\ell)} := \phi_\ell\left(z^{(\ell-1)}\right) :=
    \begin{cases}
        \left( \phi_{k,\alpha,\ell,1}^{(2)}\left(z^{(\ell-1)}\right) , \dots , \phi_{k,\alpha,\ell,\lceil d/2^\ell\rceil}^{(2)}\left(z^{(\ell-1)}\right) \right), \quad &\text{if } \big\lceil\tfrac{d}{2^\ell}\big\rceil \text{ is even}, \\
        \left( \phi_{k,\alpha,\ell,1}^{(2)}\left(z^{(\ell-1)}\right) , \dots , \phi_{k,\alpha,\ell,\lceil d/2^\ell\rceil-1}^{(2)}\left(z^{(\ell-1)}\right), z^{(\ell-1)}_{2\lceil d/2^\ell\rceil-1} \right), \quad &\text{if } \big\lceil\tfrac{d}{2^\ell}\big\rceil \text{ is odd}.
    \end{cases}
\]
Therefore, the resulting function $\phi_{k,\alpha}^{(d)}$ is expressed as a composition $\phi_D \circ \dots \circ \phi_1$ of these functions. 

Since $\operatorname{Proj}_j^\ell \in \nN \left(2,0,1  \right)$, by the composition property~\cref{proposition:properties_neural_netowrk}({\romannumeral 2}) we conclude that, for each $j$,
$$
    \phi_{k,\alpha,\ell,j}^{(2)} =\phi_{k,\alpha}^{(2)}\circ\operatorname{Proj}_j^\ell\in \nN \left( 6, 2,\frac{3k^\alpha}{2|a_2|}\right),
$$ 
and, therefore, the concatenation property~\cref{proposition:properties_neural_netowrk}({\romannumeral 3}) implies that
$$
\phi_\ell\in \nN \left(6 \Big\lceil \frac{d}{2^\ell} \Big\rceil, 2,\frac{3k^\alpha}{2|a_2|}\right).
$$
Since the constructed neural network $\phi_{k,\alpha}^{(d)}$ is a composition of $D = \lceil\log_2 d\rceil$ neural networks, by using the composition property once again we conclude
$$
\phi_{k,\alpha}^{(d)} \in \nN \left(6 \Big\lceil \frac{d}{2} \Big\rceil, 2 \lceil\log_2 d\rceil,\left[\frac{3}{2}\frac{k^\alpha}{|a_2|}\right]^{\lceil\log_2 d\rceil} \right).
$$

The error estimation follows from~\eqref{lem3:err_estim}. We let $\eps_k:= C_{\star} \,k^{-\alpha}$ and we denote by $P^{(\ell)}_j$, $\ell=1,\dots,D$, the \emph{exact} product of the inputs to $y^{(\ell)}_j$, so $P^{(0)}_j=x_j$ and $P^{(\ell)}_j=P^{(\ell-1)}_{2j-1}P^{(\ell-1)}_{2j}$. We define the error at layer $\ell$ as
\[ 
\eps_{k,\ell}:=\max_j \bigl| y^{(\ell)}_j - P^{(\ell)}_j \bigr|.
\]
At layer 1, we have that $\eps_{k,1}\le \eps_k$. For $\ell\ge 2$, by using that $|P^{(\ell-1)}_{i}|\le 1$ and $|y^{(\ell-1)}_{i}|\le 1$, we have that
\begin{align*}
|y^{(\ell)}_j-P^{(\ell)}_j | &= \Big| \phi_k^{(2)}\left(y^{(\ell-1)}_{2j-1},y^{(\ell-1)}_{2j}\right) - P^{(\ell-1)}_{2j-1}P^{(\ell-1)}_{2j} \Big| 
\\ &\le \big|y^{(\ell-1)}_{2j-1}y^{(\ell-1)}_{2j} - P^{(\ell-1)}_{2j-1}P^{(\ell-1)}_{2j} \big| + \eps_k 
\\ &\le |P^{(\ell-1)}_{2j-1}||y^{(\ell-1)}_{2j} -P^{(\ell-1)}_{2j}| + |y^{(\ell-1)}_{2j}||y^{(\ell-1)}_{2j-1}-P^{(\ell-1)}_{2j-1}| + \eps_k
\\ &\le 2 \eps_{k,\ell-1} + \eps_k.
\end{align*}
Thus, we have the recursive relation $\eps_{k,\ell}\le 2\eps_{k,\ell-1}+\eps_k$, with $\eps_{k,1}=\eps_k$, that gives the exponential bound
\[
\eps_{k,\ell} \le (2^\ell-1)\eps_k.
\]
By taking the maximal one ($\ell=D=\lceil\log_2 d\rceil$), we conclude the proof
\[
\max_{ x\in[-1,1]^d}\bigl|x_1\cdots x_d - \phi_{k,\alpha}^{(d)}( x)\bigr|\le (2^{\lceil\log_2 d\rceil}-1)C_{\star}\,k^{-\alpha}.
\]
\end{proof}

\subsection{Approximation of a regular function}

We now turn to approximating regular functions in the $\mathrm{Lip}_r$ class (see~\cref{def:lipschiotz-regular}). The strategy proceeds in three steps: first, we construct local Taylor polynomials of the target function; then, we approximate these polynomials using norm constrained neural networks; finally, we combine the local approximations via a partition of unity to obtain a global approximation over the domain of interest.

\begin{lemma}\label{lemma:approximation_functions}
Let $d\in\N$ and $r=m+\beta$ with $m\in\N$ and $\beta\in(0,1]$. We assume that the activation function $\sigma$ satisfies Assumption \ref{ass:sigma_exp}. We fix $\alpha,\gamma>0$ and $k\in\N$, $k\ge k_0:=\max\{1,(2/\rho)^{2/\alpha}\}$. Then, for any $f\in\Lip_{r}$, there exists a neural network $\phi_{k,\alpha}\in\nN(W,L,K)$, with
\[
W := W_{d,\gamma,m} = C_W(d,m)\,k^{d\gamma},\quad
L := L_{d} = 2 \lceil \log_2 d \rceil,\quad
K := K_{\alpha,d,\gamma,m} = C_K(d,m) k^{\alpha\lceil\log_2 d\rceil+d\gamma},
\]
satisfying the following approximation error estimate
\[
    \max_{x\in[0,1]^d} \bigl| f( x)-\phi_{k,\alpha}( x)\bigr| \le C \, \left(k^{-\gamma(m+\beta)}+k^{-\alpha}\right).
\]
\end{lemma}

\begin{proof}
We partition the unit cube $[0,1]^d$ into a collection of smaller cubes, each with the side length $h = k^{-\gamma}$. We denote this collection by $\{Q_j\}_{j=1}^{N_{\gamma}}$, where the number of cubes $N_{\gamma}$ is equal to $N_{\gamma} = h^{-d} = k^{\gamma d}.$

The multi-variable Taylor expansion of a function $f(x)$ at point $x_j$ center of cube $Q_j$ is
\begin{equation*} \label{eq:Taylor_exp}
     f(x) = \sum_{|\mathbf{s}| \leq m} \frac{\partial^{\mathbf{s}} f(x_j)}{\mathbf{s}!}(x-x_j)^{\mathbf{s}} + R_m(x,x_j),
\end{equation*}
where the Lagrange remainder term is bounded by 
\begin{equation} \label{eq:Rm_bound}
    |R_m(x,x_j)|  \le \frac{1}{m!}\sum_{|\mathbf{s}|=m} \frac{m!}{\mathbf{s}!} \sup_{\xi \in [x,x_j]} \bigl|\partial^{\mathbf{s}} f(\xi) - \partial^{\mathbf{s}} f(x_j)\bigr| \left( \bigl\| x-x_j \bigr\|_\infty\right)^m \leq \sum_{|\mathbf{s}|=m} \frac{1}{2^{m+\beta}\mathbf{s}!}h^{m+\beta} =:  C_1(d,m)k^{-\gamma(m+\beta)},
\end{equation}
due to the H\"older-continuity assumption on the derivatives of $f$ and the fact that the center of each cube satisfies $\|x-x_j\|_\infty \le h/2$ by the construction. 

We now introduce a partition of unity $\{\rho_j\}_{j=1}^{N_\gamma}$ subordinate to a cubes partition $\{Q_j\}_{j=1}^{N_\gamma}$ by using neural networks.
\begin{enumerate}[label=(\roman*)]
    \item We define the \emph{local bump functions} by setting
    \begin{equation} \label{def:nu}
        \eta_j(x) := \prod_{i=1}^d \text{ReLU}\bigl(1 - k^\gamma|x_i - (x_j)_i|\bigr),
   \end{equation} 
        where $x_i$ and $(x_j)_i$ denote the $i$-th component of $x$ and $x_j$, respectively.
        \item We normalize each function as follows
        \begin{equation} \label{def:rhoj}
            \rho_j(x) := \frac{\eta_j(x)}{\sum_{\ell=1}^{N_\gamma} \eta_\ell(x)}.
         \end{equation}
    \end{enumerate}
    It is easy to verify that a set of functions $\{\rho_j\}_{j=1}^{N_\gamma}$ is a partition of unity, specifically, that $\rho_j \ge 0$, their support satisfies $\text{supp}(\rho_j) \subset Q_j$ and the sum is one at every point of the domain:
    \[
        \sum_{j=1}^{N_\gamma} \rho_j(x) = 1 \qquad \text{for all } x \in [0,1]^d.
    \]
    
    For each multi-index $\mathbf{s}$ with $|\mathbf{s}| \le m$, we construct an approximating neural network, denoted by $\phi_{k,\mathbf{s},j}(x)$, by using~\cref{lemma:approx_xy_d}. Furthermore, the corresponding error is bounded by
    \[
        \max_{x \in Q_j} \bigl|(x-x_j)^{\mathbf{s}} - \phi_{k,\mathbf{s},j}(x)\bigr| \leq C_2(d,\mathbf{s}) k^{-\alpha},
    \]
    where $C_2(d,\mathbf{s})$ is the error constant from \cref{lemma:approx_xy_d}. On each cube $Q_j$, we construct the neural network
    \[
        \phi_{k,j}(x) = \sum_{|\mathbf{s}| \leq m} \frac{\partial^{\mathbf{s}} f(x_j)}{\mathbf{s}!}\phi_{k,\mathbf{s},j}(x).
    \]
    Since $\phi_{k,\mathbf{s},j}\in\nN\left(6\left\lceil\frac{d}{2}\right\rceil,2\lceil\log_2 d\rceil,\left[\frac{3}{2}\frac{k^\alpha}{|a_2|}\right]^{\lceil\log_2 d\rceil}\right)$, the width of $\phi_{k,j}$ is given by $6\left\lceil\frac{d}{2}\right\rceil$ times the number of monomials of $d$ variables with $|\mathbf s| \le m$, that is
    \begin{equation}\label{eq:partial_width}
        6\left\lceil\frac{d}{2}\right\rceil \cdot \sum_{n=0}^m \binom{n+d-1}{d-1} = 6\left\lceil\frac{d}{2}\right\rceil \cdot \binom{m+d}{d}
    \end{equation}
    from the \emph{hockey-stick} (or Pascal) identity, its depth is $2\lceil\log_2 d\rceil$, and the constraint on the weight norm is
    \[
         \sum_{|\mathbf{s}| \leq m} \left| \frac{\partial^{\mathbf{s}} f(x_j)}{\mathbf{s}!} \right| \left[\frac{3}{2}\frac{k^\alpha}{|a_2|}\right]^{\lceil\log_2 d\rceil} \leq C_0\binom{m+d}{d} k^{\alpha\lceil\log_2 d\rceil}, \quad \text{with} \quad C_0 = \left[\frac{3}{2|a_2|}\right]^{\lceil\log_2 d\rceil}.
    \]
    Moreover, the error in $Q_j$ can be bounded by using the Taylor remainder and the monomial approximation
    \[
        \bigl|f(x) - \phi_{k,j}(x)\bigr| \le \bigl|R_m(x,x_j)\bigr| + \sum_{|\mathbf{s}| \leq m} \frac{|\partial^{\mathbf{s}} f(x_j)|}{\mathbf{s}!}\bigl|(x-x_j)^{\mathbf{s}} - \phi_{k,\mathbf{s},j}(x)\bigr| \leq C_1(d,m)k^{-\gamma(m+\beta)} + C_3(d,m)k^{-\alpha},
    \]
    where $C_3(d,m)=\sum_{|\mathbf{s}| \leq m}\frac{C_2(d,\mathbf{s})}{\mathbf{s}!}$.
    
    We now define the global network by gluing together all $\phi_{k,j}$ with the partition of unity. In other words, we define
    \[
        \phi_{k,\alpha}(x) := \sum_{j=1}^{N_\gamma} \rho_j(x)\phi_{k,j}(x).
    \]
    The hyper-parameters of this network are given by applying~\cref{proposition:properties_neural_netowrk}. Specifically, for the width we multiply~\eqref{eq:partial_width} by the number of cubes, obtaining
    \[
        W_{d,\gamma,m} = 6\left\lceil\frac{d}{2}\right\rceil \cdot \binom{m+d}{d} \cdot N_\gamma =: C_W(d,m) k^{d\gamma},
    \]
    while the constraint on the weight norm is
    \[
        K_{\alpha,d,\gamma,m} = C_0\binom{m+d}{d} k^{\alpha\lceil\log_2 d\rceil} N_\gamma =: C_K(d,m) k^{\alpha\lceil\log_2 d\rceil+d\gamma}.
    \]
    The depth does not change ($L_d = 2 \lceil \log_2 d \rceil$) because each $\phi_{k,\mathbf{s},j}$ has the same depth (by~\cref{lemma:approx_xy_d}). The total approximation error, on the other hand, satisfies
    \[
        \bigl|f(x) - \phi_{k,\alpha}(x)\bigr| \leq \sum_{j=1}^N \rho_j(x)\, \bigl|f(x) - \phi_{k,j}(x)\bigr| \leq \sum_{j=1}^N \rho_j(x)\bigl(C_1k^{-\gamma(m+\beta)} + C_3k^{-\alpha}\bigr) \leq C(k^{-\gamma(m+\beta)} + k^{-\alpha}),
    \]
    as a consequence of the fact that $\sum_{j=1}^{N_\gamma} \rho_j(x)= 1$ for all $x \in [0,1]^d$, concluding the proof.
\end{proof}

\begin{remark}\label{remark:balancing_powers}
The proof reveals that the optimal choice of $\gamma > 0$ arises from balancing the approximation errors, yielding the asymptotic relation
\[
    k^{-\gamma r} \sim k^{-\alpha} \qquad \text{as } k \to \infty.
\]
By setting $\gamma = \alpha/r$, where $r = m + \beta$, we obtain the approximation error estimate
\[
    \max_{x\in[0,1]^d}|f(x)-\phi_{k,\alpha}(x)| \le 2Ck^{-\alpha},
\]
with hyperparameters given by
\[
W = C_W(d,m)k^{ d\alpha/r},\qquad
L = 2\lceil\log_2 d\rceil,\qquad
K = C_K(d,w) k^{\alpha(\lceil\log_2 d\rceil+d/r)}.
\]
For sufficiently smooth functions, i.e. $m \gg 1$, the optimal parameter satisfies $\gamma < 1$. This implies that the neural network $\phi_{k,\alpha}$ is narrow, demonstrating how smoothness properties of the target function directly influence the neural network architecture.
\end{remark}

We are now prepared to demonstrate the upper bound in the main result, as stated in the first point of~\cref{thm:main}.

\begin{proof}[Proof of~\cref{thm:main}({\romannumeral 1})]
We apply~\cref{lemma:approximation_functions} with the choice $\gamma=\alpha/r$ to obtain, for any function $f\in\Lip_r$, that there exists a neural network $\phi_{k,\alpha}\in\nN(W,L,K)$ such that
\[
    \|f-\phi_{k,\alpha}\|_{\infty,[0,1]^d} \le C\left(k^{-\alpha}+k^{-\alpha}\right) \le 2C \,k^{-\alpha},
\]
with $W \ge C_W k^{d\gamma} = C_W k^{d\alpha/r}$, $L\ge2\lceil\log_2 d\rceil$ and $K \ge C_K k^{\alpha(\lceil\log_2 d\rceil+d/r)}\ge C_K k^\alpha$. It follows from such architecture parameters that
\[
    W \ge 6\left\lceil\frac{d}{2}\right\rceil \binom{m+d}{d} k_0^{d\alpha/r} =: c_1, \qquad K \ge \frac{1}{6\lceil d/2 \rceil} \left[ \frac{3k_0^\alpha}{2|a_2|} \right]^{\lceil\log_2d\rceil} W =: c_2W. 
\]
If we choose any $k\in\N$, $k\ge\max\{1,(2/\rho)^{2/\alpha}\}$, satisfying
\[
    k \ge \max\left\{\left(\frac{W}{C_W}\right)^{r/(d\alpha)}, \left(\frac{K}{C_K}\right)^{\,1/\alpha}\right\},
\]
then
\[
    \|f-\phi_{k,\alpha}\|_{\infty,[0,1]^d} \le 2Ck^{-\alpha} \le C'\min\Big\{W^{-r/d}, K^{-1}\Big\} \le C_1W^{-r/d}+C_2K^{-1},
\]
uniformly over $\phi\in\nN(W,L,K)$ and $f\in\Lip_r$. This proves~\eqref{eq:main_upper_smooth}.
\end{proof}

\begin{remark}
All results in this section extend to weaker regularity assumptions on $\sigma$, see~\cref{subsec:lower_regularity}. For clarity of exposition, however, we have worked under Assumption \ref{ass:sigma_exp}.
\end{remark}

\section{Lower bound} \label{sec:lower_bound}

In this section, we establish the lower bound in~\cref{thm:main}({\romannumeral 2}) by analyzing the packing number of neural network evaluations. In particular, our approach estimates the packing number of evaluation sets from $\Phi:=\nN(W,L,K)$ on a finite grid, by utilizing Rademacher complexity techniques following the framework of~\cite{maiorov-ratsaby,JIAO2023249} combined with our estimates from~\cref{lemma:rad_compl}.

\smallskip

We fix a smooth bump $\psi\in C^\infty(\R^d;[0,1])$ such that
\[
    \psi(0)=1,\qquad \psi(x)=0\quad\text{if }\|x\|_\infty\ge \tfrac14.
\]
We select a scaling constant $C_{\psi,r}$ such that $C_{\psi,r}\,\psi$ belongs to $\Lip_r(\R^d)$ with unit $\Lip_r$-norm. This scaling ensures that all derivatives up to order $m$ and the $\beta$-H\"older modulus of order $m$ are properly normalized (cf.~\cref{def:lipschiotz-regular}). For each $N\in\N$, we construct a $d$-dimensional grid
\[
\Delta_N:=\Bigl\{\frac{n}{N}\ :\ n\in\{0,1,\dots,N-1\}^d\Bigr\}\subset[0,1]^d,
\]
containing $m:=|\Delta_N|=N^d$ points. Associated with each grid point, we define the scaled and translated bump functions as follows:
\[
\psi_{n}(x):= \frac{C_{\psi,r}}{N^r} \, \psi(Nx-n).
\]
These functions, which have disjoint supports, possess two crucial properties:
\begin{itemize}
    \item $\psi_n\in\Lip_r$,
    \item $\psi_n\left(\frac{n}{N}\right)=C_{\psi,r}N^{-r}$ and $\psi_n\left(\frac{n'}{N}\right)=0$ for any $n'\neq n$.
\end{itemize}
For each binary sequence $a=(a_n)_{n}\in\{\pm1\}^{\{0,\dots,N-1\}^d}$
, we define the function
\[
h_a(x):=\sum_{n\in\{0,\dots,N-1\}^d} a_n\,\psi_n(x)\in \Lip_r.
\]
When evaluated at the points of the grid $\Delta_N$, this function takes the value $C_{\psi,r}N^{-r}$ multiplied by the term $a_n$ indexed as the corresponding point of the grid. In other words, we have
\[
    h_a(\Delta_N)=\left\{C_{\psi,r}N^{-r} a_n \ : \ n\in\{0,1,\dots,N-1\}^d \right\}.
\]
To construct a subset of $\Lip_r$-function evaluations on the grid whose elements are separated by at least a positive threshold, we employ a standard coding argument recalled here for convenience; see~\cite[Lemma~3.8]{JIAO2023249}. If $m\ge8$, there exists a subset $\mathcal B_N\subset\{\pm1\}^{\{0,\dots,N-1\}^d}$ with $|\mathcal B_N|\ge 2^{m/4}$ such that any two distinct sequences $a,a'\in\mathcal B_N$ differ in at least $\lfloor m/8\rfloor$ coordinates, where $m=N^d$. We quantify separation using the empirical metric
\[
    \rho_2(x,y):=m^{-1/2}\,\|x-y\|_2\qquad (x,y\in\R^m).
\]
For distinct $a, a'\in\mathcal B_N$, the distance between the evaluations of $h_a$ and $h_{a'}$ on the grid satisfies
\[
    \rho_2\left(h_a(\Delta_N),h_{a'}(\Delta_N)\right) = C_{\psi,r} \, N^{-r} \frac{2}{\sqrt{m}}\sqrt{\#\{n \ : \ a_n\neq a'_n\}} \ge \frac{1}{2} C_{\psi,r} N^{-r}.
\]
In other words, $\{ h_a(\Delta_N) \ : \ a\in\mathcal{B}_N \}$ is a $\frac{1}{2}C_{\psi,r}N^{-r}$-packing of $\{ h_a(\Delta_N) \ : \ a\in\{\pm1\}^{\{0,\dots,N-1\}^d} \}$, and hence
\[ 
\log \cN_p\left(\{h_a(\Delta_N)\ :\ a\in\mathcal B_N\}, \rho_2, \tfrac{C_{\psi,r}}{2}N^{-r}\right) \ge \log |\mathcal{B}_N| \ge \frac{m}{4}\log 2 = \frac{N^d}{4}\log 2.
\]
Next, we bound the packing number of the set of neural network evaluations on the grid. We define
\[
S_\Phi(\Delta_N):=\{(\phi(x))_{x\in\Delta_N}:\phi\in\Phi\}\subset\R^m
\]
as the set of all evaluation vectors obtained by evaluating networks in $\Phi$ on the grid points. A Sudakov-type entropy bound (see~\cite[Corollary 4.14]{ledoux1991probability} or~\cite[Theorem 13.4]{boucheron-lugosi-massart}) relates the packing number to the Rademacher complexity (\cref{def:rad}): for some universal $C>0$ and all $\eps > 0$,
\begin{equation}\label{eq:sudakov}
\log \cN_p\left(S_\Phi(\Delta_N),\rho_2,\eps\right) \le \frac{Cm\,\mathcal R_m\left(S_\Phi(\Delta_N)\right)^2}{\eps^2} \log\left(2+\frac{1}{\sqrt{m} \mathcal R_m(S_\Phi(\Delta_N))}\right).
\end{equation}
We now estimate the Rademacher complexity from both sides. The upper bound from~\cref{lemma:rad_compl}({\romannumeral 1}) with $B=1$ gives
\begin{equation}\label{eq:R-upper}
\mathcal R_m\left(S_\Phi(\Delta_N)\right) \le \frac{K}{\sqrt m} \sqrt{2\left(L+1+\log(d)\right)}.
\end{equation}
while the lower bound from~\cref{lemma:rad_compl}({\romannumeral 3}), valid for 
sufficiently large $m$, yields
\begin{equation}\label{eq:R-lower}
\mathcal R_m\left(S_\Phi(\Delta_N)\right) \ge \frac{\sigma'(0)}{8M}\,K\frac{s^2}{m}.
\end{equation}
By substituting the upper bound~\eqref{eq:R-upper} into the main term of~\eqref{eq:sudakov}, we obtain
\begin{equation}\label{eq:pack-upper}
\log \cN_p\left(S_\Phi(\Delta_N),\rho_2,\eps\right) \le \frac{C' K^2 L}{\eps^2}\log\left(2+\frac{1}{\sqrt{m} \, \mathcal R_m(S_\Phi(\Delta_N))}\right),
\end{equation}
where $C'>0$ is a constant satisfying $2C(L+1+\log(d))\le C'L$. To control the logarithmic factor, we apply the lower bound~\eqref{eq:R-lower}, which gives an upper bound on the argument of the logarithm. For sufficiently large $m$,
\begin{equation}\label{eq:log-factor}
\log\left(2+\frac{1}{\sqrt{m}\,\mathcal R_m(S_\Phi(\Delta_N))}\right) \le \log \left(2+\frac{8M}{\sigma'(0)}\frac{1}{K}\frac{\sqrt{m}}{s^2}\right) \le \log \left(2+\frac{8M}{\sigma'(0)}\frac{\sqrt{m}}{s^2}\right).
\end{equation}

\begin{proof}[Proof of \cref{thm:main}({\romannumeral 2})]
We proceed by contradiction. We suppose that for some $\eps>0$,
\[
\mathcal E(\Lip_r,\Phi)< \eps.
\]
We set $\eps = \frac{C_{\psi,r}}{8}N^{-r}$, where $N\in\N$ will be chosen later. Then for each $a\in\mathcal B_N$, there exists $\phi_a\in\Phi$ satisfying
\[
    \|\phi_a-h_a\|_\infty = \sup_{x\in[0,1]^d}|\phi_a(x)-h_a(x)| \le \eps = \frac{C_{\psi,r}}{8}\,N^{-r}.
\]
Evaluating at the grid points of $\Delta_N$ and applying the triangle inequality, we obtain for all distinct $a,a'\in\mathcal B_N$:
\begin{align*}
\rho_2\left(\phi_a(\Delta_N),\phi_{a'}(\Delta_N)\right)&\ge\rho_2\left(h_a(\Delta_N),h_{a'}(\Delta_N)\right) - \rho_2\left(\phi_a(\Delta_N)-h_a(\Delta_N),0\right)-\rho_2\left(\phi_{a'}(\Delta_N)-h_{a'}(\Delta_N),0\right) \\
&\ge \frac{C_{\psi,r}}{2}N^{-r} - \|\phi_a-h_a\|_\infty - \|\phi_{a'}-h_{a'}\|_\infty \\
&\ge \frac{C_{\psi,r}}{2}N^{-r} - \frac{C_{\psi,r}}{8}N^{-r} - \frac{C_{\psi,r}}{8}N^{-r} = \frac{C_{\psi,r}}{4}N^{-r}.
\end{align*}
Thus the set $\{ \phi_a(\Lambda_N) \ : \ a\in\mathcal{B}_N \}$ forms a $\frac{1}{4}C_{\psi,r}N^{-r}$-packing of $S_\Phi(\Delta_N)$, whence
\[
\log \cN_p\left(S_\Phi(\Delta_N),\rho_2,\frac{C_{\psi,r}}{4}N^{-r}\right) \ge \log \cN_p\left(\{ \phi_a(\Lambda_N) \ : \ a\in\mathcal{B}_N \},\rho_2,\frac{C_{\psi,r}}{4}N^{-r}\right) \ge \log |\mathcal B_N| \ge \frac{N^d}{4}\,\log 2.
\]
On the other hand, applying~\eqref{eq:pack-upper} with $\eps=\frac{C_{\psi,r}}{4}N^{-r}$ and substituting the bound~\eqref{eq:log-factor} yields
\[
\frac{N^d}{4} \log 2 \le \frac{16 C' K^2 L}{C_{\psi,r}^2N^{-2r}}\log\left( 2 + \frac{8M}{\sigma'(0)}\frac{\sqrt{m}}{s^2} \right) = \frac{C''_\sigma K^2L}{N^{-2r}},
\]
where the constant $C''_\sigma>0$ depends on the activation function $\sigma$. Rearranging gives
\[
N^{d-2r} \le 4C''_\sigma K^2L.
\]
Since $d>2r$, choosing $N=\max\left\{8^\frac{1}{d},\left(5C''_\sigma K^2L\right)^{1/(d-2r)}\right\}$ produces a contradiction. We conclude that
\[
    \mathcal E(\Lip_r,\Phi) \ge \frac{C_{\psi,r}}{8}N^{-r} \gtrsim \left(K^2L\right)^{-\frac{r}{d-2r}},
\]
as claimed.
\end{proof}

\begin{remark}
    The lower bound obtained here matches that of~\cite[Theorem 3.2]{JIAO2023249}. The distinction is that our constant $C''_\sigma$ depends explicitly on the choice of activation function.
\end{remark}

\section{Upper bound in the non-deterministic case} \label{sec:upper_bound_nd}

In this section, we construct neural networks with randomly sampled weights to approximate functions. We establish probabilistic upper bounds on the approximation error, revealing that a trade-off between error bound and norm constraints on the network weights persists.

\begin{lemma}\label{lemma:approx_random_1}
Let $\alpha>0$ and let $\sigma:\R \rightarrow\R$ be an activation function satisfying Assumption \ref{ass:sigma_exp}. For any $k\in\N$ with $k\ge k_0:=\max\{1,\rho^{-2/\alpha}\}$, there exists a random neural network $\phi_{k, \alpha}\in\nN(2k, 2, K_{k,\alpha})$ with a norm constraint
\begin{equation} \label{lem6:norm_const}
    K_{k,\alpha} = 2\frac{k^{\alpha}}{|a_2|},
\end{equation}
such that $\phi_{k, \alpha} ([0,1])\subseteq [0,1]$ and $\phi_{k, \alpha}(x)$ approximates $f(x) = x^2$ with the following probabilistic guarantee: for all $\eps > \eps_0 := \frac{2 M k^{-\alpha}}{|a_2|}$,
\begin{equation}\label{eq:x2aub}
\mathbb P \left(\, \bigl|\phi_{k, \alpha}(x)-x^2\bigr| \le \eps \right)  \geq 1-2 \exp\left(\frac{-k (\eps-\eps_0)^2}{ C }\right), \qquad C := 2 \left(\frac{1}{3} + \frac{(2M)^2 k^{-2\alpha}}{a_2^2}\right) + \tfrac{2}{3} \left(  1+ \frac{6M\,k^{-\alpha}}{|a_2|} \right)^2.
\end{equation}
\end{lemma}

\begin{proof}
We define the symmetric $2$-neuron block
\[
\mathcal S_{\mathrm{sym}}(w;x) := \sigma( w x ) + \sigma( - w x ).
\]
To ensure the activation function expansion remains valid for $x \in [0,1]$, we require the weights to satisfy $w \le k^{-\alpha/2} \leq \rho$, which holds for $k \ge k_0$. By symmetry, only even powers appear in the Taylor expansion:
\begin{equation}\label{eq:Ssym-expansion}
\mathcal S_{\mathrm{sym}}(w;x) = 2a_2 w^2x^2 + R(w;x), \qquad |R(w;x)| \le 2M\,k^{-2\alpha}.
\end{equation}
Let $\{w_i\}_{i=1}^k$ be independent random variables with distribution $k^{-\alpha/2}\sqrt{U_i}$, where $U_i \sim \text{Unif}([0,1])$ are i.i.d. random variables. We define the random network
\[
\Phi_{k,\alpha}(x):=\frac{1}{a_2\,k^{1-\alpha}}\sum_{i=1}^k \mathcal S_{\mathrm{sym}}(w_i;x).
\]
This function belongs to $\nN(2k, 1, K_{k,\alpha})$ with norm constraint~\eqref{lem6:norm_const}. 

By using~\eqref{eq:Ssym-expansion} and the values of averages $\mathbb E\bigl[w_i^2\bigr] = k^{-\alpha}$ and $ \mathbb E [U_i]=k^{-\alpha} \tfrac12,$ we compute the expected value
\begin{equation*}
\mathbb{E}\bigl[\Phi_{k,\alpha}(x)\bigr] = \frac{1}{a_2 k^{1-\alpha}} \sum_{i=1}^k \mathbb{E}\bigl[\mathcal{S}_{\mathrm{sym}}(w_i; x)\bigr] = \frac{k}{a_2 k^{1-\alpha}} \Bigl( a_2 k^{-\alpha}x^2  + \mathbb{E}[R(w_1; x)] \Bigr) = x^2 + \frac{k^\alpha}{a_2} \mathbb{E}[R(w_1; x)].
\end{equation*}
Since the error term is bounded $|\mathbb{E}[R]| \le 2M k^{-2\alpha}$, the expected value $\mathbb{E}\bigl[\Phi_{k,\alpha}(x)\bigr]$ approximates $x^2$ with error
\[
\bigl\lvert \mathbb{E}[\Phi_{k,\alpha}(x)] - x^2 \bigr\rvert \le \eps_0, \qquad \text{where} \quad \eps_0 := \frac{2M k^{-\alpha}}{|a_2|}.
\]
Now, we define the centered random variables
\[
Y_i(x) := \frac{1}{a_2 k^{-\alpha}}\left(\mathcal S_{\mathrm{sym}}(w_i;x)-\mathbb E \bigl[ \mathcal S_{\mathrm{sym}}(w_i;x) \bigr] \right),
\]
so that $\Phi_{k,\alpha}(x)-\mathbb E\bigl[ \Phi_{k,\alpha}(x) \bigr]=\frac 1k \sum_{i=1}^k Y_i(x)$. These are zero mean bounded i.i.d. random variables:
\begin{align*} 
    |Y_i(x)|  &= \frac{1}{|a_2|k^{-\alpha}} \Big| 2 a_2 \left(x^2 ( w_i^2 - \mathbb{E}[w_i^2])  \right) + R(w_i,x) - \mathbb{E}[R(w_i,x)]\Big| \\
    &\leq x^2 + \frac{4M\,k^{-\alpha}}{|a_2|} \leq C_1, \qquad C_1: = C_1(M, a_2,k,\alpha) = \left( 1+ \frac{4M\,k^{-\alpha}}{|a_2|} \right).
\end{align*}
The variance of $Y_i(x)$ is also bounded
\begin{equation*}
\mathrm{Var}[Y_i(x)] = \frac{1}{a_2^2 k^{-2\alpha}} \Bigl( 4a_2^2 x^4 \mathrm{Var}[w_i^2] + \mathrm{Var}[R(w_i; x)] \Bigr) \le \frac{x^4}{3} + \frac{(2M)^2 k^{-2\alpha}}{a_2^2} \le C_2, \qquad C_2 := \frac{1}{3} + \frac{(2M)^2 k^{-2\alpha}}{a_2^2},
\end{equation*}
where we have used $\text{Var}[w_i^2] = \frac{k^{-2\alpha}}{12}$. Bernstein's inequality yields, for all $\delta>0$,
\[
\mathbb{P}\Bigl( \bigl\lvert \Phi_{k,\alpha}(x) - \mathbb{E}[\Phi_{k,\alpha}(x)] \bigr\rvert > \delta \Bigr) \le 2\exp\left( \frac{-k\delta^2}{2C_2 + \frac{2}{3} C_1 \delta} \right).
\]
We note that, if $\delta > C_1$, the probability is trivially zero as a consequence of the boundness of the random variables $Y_i$. 

For any $\eps > C_1$, the triangle inequality gives
\begin{align*}
\mathbb{P}\bigl( |\Phi_{k,\alpha}(x) - x^2| \le \eps \bigr) &\ge \mathbb{P}\bigl( |\Phi_{k,\alpha}(x) - \mathbb{E}[\Phi_{k,\alpha}(x)]| + |\mathbb{E}[\Phi_{k,\alpha}(x)] - x^2| \le \eps \bigr) \\
&\ge \mathbb{P}\bigl( |\Phi_{k,\alpha}(x) - \mathbb{E}[\Phi_{k,\alpha}(x)]| \le \eps - \eps_0 \bigr).
\end{align*}
By applying Bernstein's inequality with $\delta = \eps - \eps_0$ and setting $C := 2 C_2 + \tfrac{2}{3} (\eps_0+C_1)^2$ due to $|\Phi_{k,\alpha}(x)-\mathbb E\bigl[ \Phi_{k,\alpha}(x)] | < \eps_0,$  we obtain
\begin{equation}\label{eq:ConvProb}
\mathbb{P}\bigl( |\Phi_{k,\alpha}(x) - x^2| \le \eps \bigr) \ge 1 - 2\exp\biggl( \frac{-k(\eps - \eps_0)^2}{C} \biggr).
\end{equation}
Finally, we define $\psi(t):=\mathrm{ReLU}(t)-\mathrm{ReLU}(t-1)\in\nN(2,1,1)$ and set $\phi_{k, \alpha}:=\psi \circ \Phi_{k,\alpha} \in \nN(2k, 2, K_{k,\alpha})$. Since $\psi$ is 1-Lipschitz and $x^2 \in [0,1]$ for $x \in [0,1]$, we have
\[
|\phi_{k,\alpha}(x) - x^2| = |\psi(\Phi_{k,\alpha}(x)) - \psi(x^2)| \le |\Phi_{k,\alpha}(x) - x^2|,
\]
so the probabilistic bound is preserved.
\end{proof}

\begin{remark}
    We notice that, as $k\to\infty$, the probabilistic tolerance $\eps$ in \eqref{eq:x2aub} can be pushed towards zero, and the probability that \eqref{eq:x2aub} holds with such a small tolerance approaches one.
\end{remark}

\begin{remark}
Since the Taylor expansion of an activation function is valid only within its radius of convergence, the random variables must be constructed to remain bounded. However, as $k$ increases, the exponential convergence rate in \eqref{eq:ConvProb} arises from a law of large numbers. Thus, for alternative distributions, such as discrete or truncated normals, where the weights $\{w_i\}_{i=1}^k$ are drawn randomly such that $\pm w_i x$ stays within the radius of convergence, the same exponential bound holds, albeit with a modified constant $C$.
\end{remark}

It is also possible to generalize~\cref{lemma:approx_xy} and~\cref{lemma:approx_xy_d} to the probabilistic setting. Specifically, the construction above can be adapted to approximate $xy$ and then $x_1\cdots x_d$.

\begin{lemma} \label{lemma7}
Let $\alpha>0$ and let $\sigma:\R\rightarrow\R$ be an activation function satisfying Assumption \ref{ass:sigma_exp}. For any $k\in\N$ with $k\ge k_0:=\max\{1,(2/\rho)^{2/{\alpha}}\}$, there exists a random neural network $\phi_{k, \alpha}\in\nN(4k, 2, K_{k,\alpha})$ with
\begin{equation} \label{lem7:norm_const}
    K_{k,\alpha} = \frac{ k^{\alpha}}{|a_2|},
\end{equation}
such that $\phi_{k, \alpha} ([-1,1]^{2})\subseteq [-1,1]$ and $\phi_{k, \alpha}(x,y)$ approximates $f(x,y) = xy$ with the following probabilistic guarantee: for all $\eps > \eps_0 := \tfrac{ 16 M k^{-\alpha}}{|a_2|}$,
\begin{equation} \label{eq:prob_est_lemmaXY}
\mathbb{P}\bigl(\lvert \phi_{k,\alpha}(x,y) - xy \rvert \le \eps\bigr) \ge 1 - 2\exp\biggl(\frac{-k(\eps - \eps_0)^2}{B}\biggr), \qquad B := 2 \left(\frac{1}{3} + \frac{256M^2 k^{-2\alpha}}{a_2^2}\right) + \tfrac{2}{3} \left(1 + \frac{48 M k^{-\alpha}}{|a_2|} \right)^2 .
\end{equation}

\end{lemma}

\begin{proof}
We define the symmetric 4-neuron bilinear block
\[
\mathcal S_{\mathrm{bil}}(w;x,y) := \sum_{\eps_1,\eps_2\in\{\pm1\}}\eps_2\,\sigma\left(\eps_1 wx+ \eps_1\eps_2 wy\right).
\]
To ensure the activation function expansion remains valid for $x,y \in [-1,1]$, we require $|w| \le k^{-\alpha/2}$, so that the argument of $\sigma$ satisfies $|\eps_1 wx + \eps_1 \eps_2 wy | \le 2k^{-\alpha/2} \le \rho$, which holds for $k \ge k_0$. By the symmetry of the block, all terms in $x$ and $y$ cancel, leaving
\begin{equation}\label{eq:Sbil-expansion}
\mathcal S_{\mathrm{bil}}(w;x,y) = 8a_2 w^2 xy + R(w;x,y), \qquad |R(w;x,y)| \le 64 M\,k^{-2\alpha}.
\end{equation}
Let $\{w_i\}_{i=1,\ldots,k}$ be independent random variables with distribution $k^{-\alpha/2}\sqrt{U_i}$, where $U_i \sim \text{Unif}([0,1])$ are i.i.d. Next, we define the random network
\[
\Phi_{k,\alpha}(x,y):=\frac{1}{4a_2 \, k^{1-\alpha}} \sum_{i=1}^k \mathcal{S}_{\mathrm{bil}}(w_i; x,y).
\]
This function belongs to $\nN(4k, 1, K_{k,\alpha})$ with norm constraint~\eqref{lem7:norm_const}. By using~\eqref{eq:Sbil-expansion} and $\mathbb{E}[w_i^2] = \frac{1}{2} k^{-\alpha}$, we compute
\begin{align*}
\mathbb{E}\bigl[\Phi_{k,\alpha}(x, y)\bigr] &= \frac{1}{4a_2 k^{1-\alpha}} \sum_{i=1}^k \mathbb{E}\bigl[\mathcal{S}_{\mathrm{bil}}(w_i; x, y)\bigr] = \frac{k}{4a_2 k^{1-\alpha}} \Bigl( 8 a_2 \cdot \tfrac{1}{2} k^{-\alpha} xy + \mathbb{E}[R(w_i; x, y)] \Bigr)
\\ &= xy + \frac{k^\alpha}{4a_2} \mathbb{E}[R(w_i; x, y)].
\end{align*}
Since $|\mathbb{E}[R]| \le 64 M\,k^{-2\alpha}$, the bias satisfies
\[ 
\bigl\lvert \mathbb{E}[\Phi_{k,\alpha}(x, y)] - xy \bigr\rvert \le \eps_0, \qquad \text{where} \quad \eps_0 := \frac{16 M k^{-\alpha}}{|a_2|}.
\]
We define the centered random variables
\[
Y_i(x, y) := \frac{1}{4a_2 k^{-\alpha}} \Bigl( \mathcal{S}_{\mathrm{bil}}(w_i; x, y) - \mathbb{E}[\mathcal{S}_{\mathrm{bil}}(w_i; x, y)] \Bigr) = \frac{1}{4a_2 k^{-\alpha}} \left(8a_2 \left(w^2 - \mathbb{E}[w^2]\right) xy + R-\mathbb{E}[R]\right),
\]
so that $\Phi_{k,\alpha}(x, y) - \mathbb{E}[\Phi_{k,\alpha}(x, y)] = \frac{1}{k} \sum_{i=1}^k Y_i(x, y)$. These are zero mean i.i.d.\ random variables, bounded by
\[
|Y_i(x, y)| \le \frac{1}{4|a_2| k^{-\alpha}} \Bigl( 4|a_2| k^{-\alpha} + 128 M k^{-2\alpha} \Bigr) \le C_1 := 1 + \frac{32 M k^{-\alpha}}{|a_2|},
\]
and with variance bounded by
\begin{align*}
\mathrm{Var}[Y_i(x, y)] &= \frac{1}{16 a_2^2 k^{-2\alpha}} \Bigl( 64 a_2^2 x^2 y^2 \mathrm{Var}[w_i^2] + \mathrm{Var}[R(w_i; x, y)] \Bigr) \\
&\le \frac{x^2 y^2}{3} + \frac{(64 M k^{-2\alpha})^2}{16 a_2^2 k^{-2\alpha}} \le C_2, \qquad C_2 := \frac{1}{3} + \frac{256M^2 k^{-2\alpha}}{a_2^2},
\end{align*}
where we have used $\operatorname{Var}[w_i^2] = \frac{k^{-2\alpha}}{12}$. Proceeding as in~\cref{lemma:approx_random_1}, Bernstein's inequality yields
\[
\mathbb{P}\bigl( |\Phi_{k,\alpha}(x, y) - xy| \le \eps \bigr) \ge 1 - 2\exp\biggl( \frac{-k(\eps - \eps_0)^2}{B} \biggr)
\]
for all $\eps > \eps_0$, where $B := 2 C_2 + \tfrac{2}{3} (\eps_0+C_1)^2$. 
Setting $\phi_{k,\alpha} := \psi \circ \Phi_{k,\alpha}$, where $\psi(t) = \mathrm{ReLU}(t+1) - \mathrm{ReLU}(t-1) - 1$, gives $\phi_{k,\alpha} \in \mathcal{NN}(4k, 2, K_{k,\alpha})$ with range $[-1,1]$. Since $\psi$ is 1-Lipschitz and $xy \in [-1,1]$ for $x, y \in [-1,1]$, the probabilistic bound is preserved.
\end{proof}

\begin{lemma}\label{lemma:random_approx_xy_d}
Let $\alpha>0$ and let $\sigma:\R\rightarrow\R$ be an activation function satisfying Assumption \ref{ass:sigma_exp}. For any $k\in\N$ with $k\ge k_0:=\max\{1,(2/\rho)^{2/{\alpha}}\}$ there exists a random neural network $\phi_{k,\alpha}^{(d)}\in\nN\left(W_d,L_d,K_{k,\alpha,d}\right)$ with
\[
W_d = 4k \biggl\lceil \frac{d}{2} \biggr\rceil, \qquad
L_d = 2 \bigl\lceil \log_2 d \bigr\rceil, \qquad
K_{k,\alpha,d} = \biggl( \frac{k^\alpha}{|a_2|} \biggr)^{\lceil \log_2 d \rceil},
\]
such that $\phi_{k,\alpha}^{(d)} \bigl( [-1,1]^d \bigr) \subseteq [-1,1]$ and $\phi_{k,\alpha}^{(d)}$ approximates the monomial $f(x_1, \dots, x_d) = \prod_{i=1}^d x_i$ with the following probabilistic guarantee: for all $\eps > (2^{\lceil \log_2 d \rceil} - 1) \eps_0$, where $\eps_0 = \frac{16 M k^{-\alpha}}{|a_2|}$,
\begin{equation} \label{eq:defP_0}
\mathbb{P}\Biggl( \biggl| \prod_{i=1}^d x_i - \phi_{k,\alpha}^{(d)}(x_1, \ldots, x_d) \biggr| \le \eps \Biggr) \ge \biggl( 1 - 2\exp\biggl( \frac{-k(\eps_d - \eps_0)^2}{B} \biggr) \biggr)^{\lceil \log_2 d \rceil},
\end{equation}
where $\eps_d := \eps / (2^{\lceil \log_2 d \rceil} - 1)$ and the constant $B$ is defined by \eqref{eq:prob_est_lemmaXY}.
\end{lemma}

\begin{proof} 
We construct $\phi_{k,\alpha}^{(d)}$ by composing random bilinear networks $\phi_{k,\alpha}^{(2)}$ from~\cref{lemma7} in a binary tree structure, following the same construction as in~\cref{lemma:approx_xy_d}.

Similarly, we set $D := \lceil \log_2 d \rceil$. At level $\ell \in \{1, \ldots, D\}$, we apply $\lceil d / 2^\ell \rceil$ independent copies of $\phi_{k,\alpha}^{(2)}$ to compute pairwise products of the outputs from level $\ell - 1$. The width at each level is at most $4k \lceil d/2 \rceil$, the depth is $2D$, and the weight norm compounds to $K_{k,\alpha} = (k^\alpha / |a_2|)^{D}$.

Next, we consider $\eps_k > \eps_0$ to be the error tolerance for each $\phi_{k,\alpha}^{(2)}$. By the same telescoping argument as in~\cref{lemma:approx_xy_d}, if each bilinear approximation satisfies $|\phi_{k,\alpha}^{(2)}(x, y) - xy| \le \eps_k$, then the accumulated error after $D$ levels satisfies
\[
\biggl| \prod_{i=1}^{d} x_i - \phi_{k,\alpha}^{(d)}(x_1, \ldots, x_d) \biggr| \le (2^D - 1) \eps_k,
\]
where $\eps_k$ is an error whose probability distribution is characterized by \eqref{eq:prob_est_lemmaXY}.
At each level $1\leq\ell \leq D$, the random weights are sampled independently from all other levels. Conditional on the outputs of level $\ell - 1$ lying in $[-1, 1]$ (which holds after the clipping operation), \cref{lemma7} guarantees that each $\phi_{k,\alpha}^{(2)}$ at level $\ell$ satisfies its error bound with a probability of at least $1 - 2\exp(-k(\eps_k - \eps_0)^2 / B)$.

Taking a union bound over the $\lceil d / 2^\ell \rceil$ networks at level $\ell$, and then bounding by treating each level as a single event, the probability that all $D$ levels succeed is at least
\[
\mathbb{P}\Biggl( \biggl| \prod_{i=1}^d x_i - \phi_{k,\alpha}^{(d)}(x_1, \ldots, x_d) \biggr| \le \eps \Biggr) \ge \biggl( 1 - 2\exp\biggl( \frac{-k(\eps_k - \eps_0)^2}{B} \biggr) \biggr)^{D}.
\]
To achieve a total error of at most $\eps$, we introduce $\eps_d := \eps / (2^{D} - 1)$. The constraint $\eps_k > \eps_0$ requires $\eps > (2^{D} - 1) \eps_0$, yielding the stated bound.
\end{proof}

To conclude this section, we show that~\cref{lemma:approximation_functions} admits a randomized analogue, from which~\cref{thm:main_pr} follows as an immediate consequence.

\begin{lemma} \label{lem:lem9}
Let $d\in \mathbb{N}$ and $m\in \mathbb{N}$ fixed. Let $\beta \in (0,1]$ and $r:=m+\beta$.
Let $\alpha>0$ and let $\sigma:\R\rightarrow\R$ be an activation function satisfying Assumption \ref{ass:sigma_exp}. For any $f\in\Lip_r$ and for any $k\in\N$ with $k\ge k_0:=\max\{1,(\rho/2)^{-2/{\alpha}}\}$, there exists a random neural network $\phi_{k,\alpha} \in \nN\left(W_{m,d,k,\alpha},L_d,K_{m,d, k, \alpha}\right)$ with parameters
\begin{equation*} \label{defWLKLemma9}
W_{m,d,k,\alpha} = 4k \biggl\lceil \frac{d}{2} \biggr\rceil \binom{m+d}{d} k^{d\frac{\alpha}{r}}, \qquad
L_d = 2 \bigl\lceil \log_2 d \bigr\rceil, \qquad
K_{m,d, k, \alpha} = C_0 \binom{m+d}{d} k^{\alpha\lceil \log_2 d \rceil+d\frac{\alpha}{r}} ,
\end{equation*}
with $C_0=|a_2|^{-\lceil \log_2 d \rceil}$, such that $\phi_{k,\alpha}$ approximates $f$ with the following probabilistic guarantee: for all $\eps>0$
\begin{equation} \label{eq:prob_est_lemma9}
\mathbb{P}\Biggl(\bigl|f(x) - \phi_{k,\alpha}(x)\bigr| \le \underbrace{ F \eps }_{\text{random estimation error}}+ 
\underbrace{G k^{-\alpha} }_{\text{deterministic estimation}}\Biggr) \geq \biggl( 1 - 2\exp\biggl( -\frac{k\eps^2 }{B} \biggr) \biggr)^{\lceil \log_2 d \rceil \binom{m+d}{d}},
\end{equation}
where $B$, $F$, and $G$ are the constant parameters defined by \eqref{eq:prob_est_lemmaXY} and  
\begin{equation*} \label{def:FandG}
F:= (2^{\lceil \log_2 d \rceil} - 1) e^d, \qquad G:= e^d \left(1 + \frac{16 M (2^{\lceil \log_2 d \rceil} - 1 )}{|a_2|}\right). 
\end{equation*}
\end{lemma}
\begin{proof}
By generalizing the proof of~\cref{lemma:approximation_functions}, we partition the cube $[0,1]^d$ into $N_\gamma := h^{-d} =  k^{\gamma d}$ axis-aligned cubes $\{Q_j\}_{j=1}^{N_\gamma}$, each of side length $h = k^{-\gamma}$. We label each cube $Q_j$ centered at $x_j \in [-1,1]^d$ by an index $j$. As we stated in \eqref{eq:Rm_bound}, the local multi-variable Taylor expansion of a function $f\in\Lip_{r}$ at point $x_j$ satisfies
\[
f(x) = T_m(x,x_j) + R_m(x,x_j), \qquad T_m(x, x_j) := \sum_{|\mathbf{s}| \le m} \frac{\partial^{\mathbf{s}} f(x_\nu)}{\mathbf{s}!} (x - x_\nu)^{\mathbf{s}}, \qquad |R_m(x,x_j)| \le C_1 k^{-\gamma(m+\beta)},
\]
where the constant $C_1$ is defined by \eqref{eq:Rm_bound}.

For each multi-variable monomial $(x-x_j)^{\boldsymbol{s}}$ with $|\boldsymbol{s}| \leq m$, we construct a random neural network $\phi_{k,\boldsymbol{s},j}(x)$ which approximates the monomial  according to the construction of  \cref{lemma:random_approx_xy_d}. Denoting the  right-hand side of \eqref{eq:defP_0} by 
\[
\mathbb{P_0}(\eps): = \biggl( 1 - 2\exp\biggl( \frac{-k(\eps (2^{\lceil \log_2 d \rceil} - 1)^{-1} - \eps_0)^2}{B} \biggr) \biggr)^{\lceil \log_2 d \rceil},
\]
the probability of an estimation error $\eps_1 > (2^{\lceil \log_2 d \rceil} - 1) \eps_0$ is bounded as follows
\[
\mathbb{P}\Biggl( \biggl| (x-x_j)^{\boldsymbol{s}} - \phi_{k,\boldsymbol{s},j}(x)\biggr| \le \eps_1 \Biggr) \ge \mathbb{P}_0(\eps_1).
\]

Similarly to the non-random case, we define a random neural network $\phi_{k,j}(x)$ approximating a Taylor approximation term $T_m(x,x_j)$ via a linear combination of already constructed neural networks $\phi_{k,\mathbf{s},j}(x)$
\[
    \phi_{k,j}(x) = \sum_{|\mathbf{s}| \leq m} \frac{\partial^{\mathbf{s}} f(x_j)}{\mathbf{s}!}\phi_{k,\mathbf{s},j}(x), \quad \text{with } \phi_{k,j}\in\nN\left(4k\left\lceil\frac{d}{2}\right\rceil \binom{m+d}{d} , 2\lceil\log_2 d\rceil , C_0\binom{m+d}{d} k^{\alpha\lceil\log_2 d\rceil} \right).
\]
The architecture parameters of $\phi_{k,j}$ are computed by applying \cref{proposition:properties_neural_netowrk}(iv) as in \cref{lemma:approximation_functions}. For simplicity, we require the error $\eps_1$ to be the same for each multi-index $\boldsymbol{s}$. Then the accumulated error of the Taylor expansion approximation is bounded as follows
\begin{align*}
 \biggl| T_m(x,x_j) - \phi_{k,j}(x) \biggr| =  \sum_{|\mathbf{s}| \leq m} \frac{|\partial^{\mathbf{s}} f(x_j)|}{\mathbf{s}!}\bigl|(x-x_j)^{\mathbf{s}} - \phi_{k,\mathbf{s},j}(x)\bigr| \leq E\eps_1, \qquad E = E(d,m) := \sum_{|\mathbf{s}| \leq m} \frac{1}{\mathbf{s}!}.
\end{align*}
We note that a random neural network $\phi_{k,\mathbf{s},j}(x)$ for a fixed $\mathbf{s}$ is constructed only once and then is used for the Taylor polynomial construction in each box. The probability bound estimate decreases as 
\begin{align*}
\mathbb{P}\Biggl( \biggl| T_m(x,x_j) - \phi_{k,j}(x) \biggr| \le E\eps_1 \Biggr) \geq \mathbb{P}_0(\eps_1)^{\binom{m+d}{d}}.
\end{align*}
At each cube $Q_j$ the random neural network $\phi_{k,j}(x)$ approximates the function $f(x)$ with a probabilistic error and the Taylor expansion error term:
\[
\bigl|f(x) - \phi_{k,j}(x)\bigr| \le  \biggl| T_m(x,x_j) - \phi_{k,j}(x) \biggr|+ \bigl|R_m(x,x_j)\bigr| \leq E \eps_1+C_1 k^{-\gamma r}.
\]

By using the partition of unity $\rho_j(x)$ defined in \eqref{def:nu} and \eqref{def:rhoj}, we similarly define 
\[
\phi_{k,\alpha}(x) := \sum_{j=1}^{N_\gamma} \rho_j(x)\, \phi_{k,j}(x) \quad \text{with } \phi_{k,\alpha}\in\nN\left(4k\left\lceil\frac{d}{2}\right\rceil \binom{m+d}{d} k^{d\gamma}, 2\lceil\log_2 d\rceil , C_0\binom{m+d}{d} k^{\alpha\lceil\log_2 d\rceil + d\gamma} \right),
\]
such that the accumulated error is bounded as follows
\[
\bigl|f(x) - \phi_{k,\alpha}(x)\bigr| \leq \sum_{j=1}^{N_\gamma} \rho_j(x)\, \bigl|f(x) - \phi_{k,j}(x)\bigr| \leq \sum_{j=1}^{N_\gamma} \rho_j(x)\bigl(E\eps_1+C_1 k^{-\gamma r} \bigr) \leq E\eps_1+C_1 k^{-\gamma r} ,
\]
with the corresponding probability bound for any $\eps_1 > (2^{\lceil \log_2 d \rceil} - 1) \eps_0$ given by
\[
\mathbb{P}\Biggl(\bigl|f(x) - \phi_{k,\alpha}(x)\bigr| \le E\eps_1+C_1 k^{-\gamma r}  \Biggr) \geq \mathbb{P}_0(\eps_1)^{\binom{m+d}{d}}.
\]
The architecture parameters of $\phi_{k,\alpha}$ are computed by applying \cref{proposition:properties_neural_netowrk}(iv) as in \cref{lemma:approximation_functions}. Substitution of $\tilde{\eps}_1:=  \eps_1 (2^{\lceil \log_2 d \rceil} - 1)^{-1} - \eps_0$ delivers
\[
\mathbb{P}\Biggl(\bigl|f(x) - \phi_{k,\alpha}(x)\bigr| \le E(2^{\lceil \log_2 d \rceil} - 1)(\tilde{\eps}_1 + \eps_0)+C_1 k^{-\gamma r}  \Biggr) \geq \mathbb{P}_0\left((2^{\lceil \log_2 d \rceil} - 1)(\tilde{\eps}_1 + \eps_0)\right)^{\binom{m+d}{d}}.
\]
The result  \eqref{eq:prob_est_lemma9} follows with the choice of the mesh parameter $\gamma = \alpha/r$, motivated by \cref{remark:balancing_powers}:
\[
\mathbb{P}\Biggl(\bigl|f(x) - \phi_{k,\alpha}(x)\bigr| \le \underbrace{ F \tilde{\eps}_1 }_{\text{random estimation error}}+ 
\underbrace{G k^{-\alpha} }_{\text{non-random error of estimation}}\Biggr) \geq \biggl( 1 - 2\exp\biggl( -\frac{k\tilde{\eps}_1^2 }{B} \biggr) \biggr)^{\lceil \log_2 d \rceil \binom{m+d}{d}},
\]
where we set $F_1:= (2^{\lceil \log_2 d \rceil} - 1) E$ and $G_1 := \frac{16 M F_1 }{|a_2|}+\sum_{|\mathbf{s}|=m} \frac{1}{2^{m+\beta}\mathbf{s}!}$ and we bound $F_1< (2^{\lceil \log_2 d \rceil} - 1) e^d =: F$ and $G_1< e^d\left(1 + \frac{16 M (2^{\lceil \log_2 d \rceil} - 1)}{|a_2|}\right)=:G$, by using the Taylor expansion of $e^{dx}=\prod_{i=1}^d \sum_{s=0}^\infty \frac{x^s}{s!}$, $x\in\R$.
\end{proof}

We are now ready to demonstrate ~\cref{thm:main_pr}.

\begin{proof}[Proof of \cref{thm:main_pr}]
We apply~\cref{lem:lem9} to obtain, for any function $f\in\Lip_r$, that there exists a random neural network $\phi_{k,\alpha}\in\nN(W,L,K)$ such that, for all $\eps>0$,
\[
    \mathbb{P}(|f(x) - \phi_{k,\alpha}(x)| \le F \eps + G k^{-\alpha} ) \geq \biggl( 1 - 2\exp\biggl( -\frac{k\eps^2 }{B} \biggr) \biggr)^{\lceil \log_2 d \rceil \binom{m+d}{d}},
\]
with $W \ge 4k\bigl\lceil \frac{d}{2} \bigr\rceil \binom{m+d}{d} k^{d\alpha/r}$, $L\ge2\lceil\log_2 d\rceil$ and $K \ge \binom{m+d}{d} |a_2|^{-\lceil \log_2 d \rceil} k^{\alpha(\lceil\log_2 d\rceil+d/r)}\ge \binom{m+d}{d} |a_2|^{-\lceil \log_2 d \rceil} k^\alpha$. It follows from such architecture parameters that
\[
    W \ge 4k\left\lceil \frac{d}{2} \right\rceil \binom{m+d}{d} k_0^{d\alpha/r} =: c_1, \qquad K \ge \frac{k_0^{\alpha\lceil\log_2d\rceil}}{4\lceil d/2 \rceil |a_2|^{\lceil\log_2d\rceil}}  W =: c_2W. 
\]
We then use the convexity of the function $(1-x)^n$ -- which implies $(1-x)^n \geq 1-nx$ for any $x \in [0,1]$ -- to get
\begin{equation*}
   \mathbb{P}\Biggl(\bigl|f(x) - \phi_{k,\alpha}(x)\bigr| \le F \eps + G k^{-\alpha} \Biggr) \geq 1 - 2\lceil \log_2 d \rceil \binom{m+d}{d} \exp\biggl( -\frac{k\eps^2 }{B} \biggr).
\end{equation*}
By introducing $C_W =4k\bigl\lceil \frac{d}{2} \bigr\rceil \binom{m+d}{d}$ and $C_K = \binom{m+d}{d} |a_2|^{-\lceil \log_2 d \rceil}$, if we choose any $k\in\N$, $k\ge\max\{1,(2/\rho)^{2/\alpha}\}$, satisfying
\[
    k \ge \max\left\{\left(\frac{W}{C_W}\right)^{r/(d\alpha)}, \left(\frac{K}{C_K}\right)^{\,1/\alpha}\right\},
\]
and we appropriately define the constants $C_1$, $C_2$ and $C_3$, then
\[
    \mathbb{P}\left(|f(x) - \phi_{k,\alpha}(x)| \le C_1 \left(W^{-r/d}+K^{-1}\right) + \eps \right) \geq 1 - C_2 \exp\left( -C_3K\eps^2 \right).
\]
uniformly over $\phi\in\nN(W,L,K)$ and $f\in\Lip_r$. This proves~\eqref{eq:main_upper_smooth_pr}.
\end{proof}

\section{Conclusions}
In this paper we have studied the approximation capacity of neural networks with an arbitrary activation function and with norm constraint on the weights. Specifically, we have computed upper and lower bounds to the approximation error of these networks for smooth function classes. We have proven the upper bound by first approximating high-degree monomials and then generalizing it to functions via a partition of unity and Taylor expansion. We have then derived the lower bound through the Rademacher complexity of neural networks, which may be of independent interest. We have also provided a probabilistic version of the upper bound by considering neural networks with randomly sampled weights and biases. Finally, we have shown that the assumption on the regularity of the activation function can be significantly weakened without worsening the approximation error, and the approximation upper bound is validated with numerical experiments.

\section{Breakdown of authors' contributions}
Arturo De Marinis has had the idea of this research work, has written the introduction and the first half of the preliminaries, has carefully revised the entire paper, and has attended the meetings for the development of the paper. Francesco Paolo Maiale has had the main idea to generalize the results of \cite{JIAO2023249} to deep neural networks with arbitrary activation functions, has written the draft of the remaining part of the paper, has carefully revised it, and has attended the meetings for the development of the paper. Anastasiia Trofimova has mainly contributed to the development of \cref{sec:upper_bound_nd}, has carefully revised the entire paper, and has attended the meetings for its development.


\section*{Acknowledgments}

Francesco Paolo Maiale acknowledge that his research was supported by funds from the Italian MUR (Ministero dell'Universit\`a e della Ricerca) within the PRIN 2022 Project ``Advanced numerical methods for time dependent parametric partial differential equations with applications'' and the PRIN-PNRR Project ``FIN4GEO''.
He is also affiliated to the INdAM-GNCS (Gruppo Nazionale di Calcolo Scientifico).

Francesco Paolo Maiale thanks Luis Eduardo Iba\~nez Perez and Pietro Sittoni (Gran Sasso Science Institute) for the helpful discussions and unwavering support throughout the finalization of this paper.

\appendix

\section{Additional results for deterministic quadratic approximation}  \label{appendix:additional}

In this section, we present additional results related to the quadratic approximation discussed in~\cref{sec:upper_bound}. These results, while not strictly necessary for the proof of~\cref{thm:main}, provide important context for our construction. Specifically, we show that smooth activation functions cannot exactly represent the hard clipping function, which motivates the use of ReLU. We also prove that the parameter scaling in~\cref{lemma:approximation_x2_regular} is optimal.

\begin{proposition}\label{prop:no_clipping}
Let $\sigma\in C^1(\mathbb R)$ and let $\nN(W,L)$ be the class of functions represented by width $W$ and depth $L$ neural networks with activation function $\sigma$ and arbitrary weights and biases. Then every $\phi\in\nN(W,L)$ belongs to $C^1(\mathbb R)$. Consequently, no $\phi\in\nN(W,L)$ can equal the hard clipping function
\[
    \psi(x):=\min\bigl\{\max\{x,0\},1\bigr\},\qquad x\in\R.
\]
\end{proposition}

\begin{proof}
Compositions of $C^1$ maps are $C^1$, and finite sums of $C^1$ maps are $C^1$. By induction over layers, every function $\phi$ of finite depth built from $\sigma\in C^1$ belongs to $C^1(\mathbb R)$. However, $\psi$ is not differentiable at $0$ and $1$, hence $\psi\notin C^1(\mathbb R)$. Therefore, no $\phi\in\nN(W,L)$ can equal $\psi$ pointwise.
\end{proof}

Next, we show that the parameter choice in~\cref{lemma:approximation_x2_regular} is optimal in the following sense: achieving uniform error $\mathcal{O}(k^{-\alpha})$ with the two-neuron network requires the parameter norm to scale as $k^{\alpha}$.

\begin{proposition}[Optimal scaling] \label{prop:opt_scaling_even2}
Let $\sigma$ be an activation function satisfying Assumption \ref{ass:sigma_exp}. We consider the two-neuron symmetric neural network
\[
    \Phi_k(x) = c\left(\sigma(wx)+\sigma(-wx)\right), \qquad x\in[0,1].
\] 
If $\max_{x \in [0,1]}|x^2 - \Phi_k(x)| = \mathcal{O}(k^{-\alpha})$, then $c = \mathcal O (k^\alpha)$ and $K = \mathcal O (k^\alpha)$.
\end{proposition}

\begin{proof}
We notice that the coefficient of $x^2$ in the Taylor expansion of $\Phi_k$ must satisfy
\[ 
    c \cdot 2a_2 w^2 = 1 + \mathcal{O}(k^{-\alpha}),
\]
because of the assumption $\max_{x \in [0,1]}|x^2 - \Phi_k(x)| = \mathcal{O}(k^{-\alpha})$. Assuming that $w = \mathcal{O}(k^{-\beta/2})$ for some $\beta > 0$, the condition above requires $c \cdot k^{-\beta} = \mathcal{O}(1)$, thus $c = \mathcal O(k^{\beta})$. On the other hand, the error $c \cdot \mathcal{O}((wx)^4)$ from the high-order terms in the Taylor expansion satisfies
\[
    c \cdot \mathcal{O}(k^{-2\beta}) = c \cdot \mathcal{O}((wx)^4) = \mathcal{O}(k^{-\alpha}), 
\]
which means that $c = \mathcal O(k^{2\beta-\alpha})$. For these two bounds to be compatible, we need $\mathcal O(k^{\beta}) = \mathcal O(k^{2\beta-\alpha})$, which is satisfied if and only if $\beta = \alpha$, giving
\[
    c = \mathcal O(k^\alpha) \qquad \text{and} \qquad K = 2c = \mathcal O(k^\alpha).
\]
\end{proof}

\section{Low regularity approximation} \label{subsec:lower_regularity}

In this section, we show that Assumption \ref{ass:sigma_exp} on $\sigma$ can be relaxed. Specifically, it is sufficient that $\sigma$ is locally Lipchitz with a locally quadratic behavior at a specific point $x_0$.

\begin{lemma}\label{lem:weak_modulus}
Let $\sigma:\R\to\R$ be locally Lipschitz and $\alpha>0$. We assume that there exist $x_0\in\R$, $0<\rho\le1$, $\gamma\neq 0$, and a nondecreasing modulus $\omega:[0,\rho]\to[0,\infty)$, with $\lim_{t\downarrow 0}\omega(t)=0$, such that
\begin{equation}\label{eq:even_modulus}
\left|\sigma(x_0+h)+\sigma(x_0-h)-2\sigma(x_0)-\gamma h^2\right| \le \omega(|h|)h^2 \qquad \text{for all } h \in[- \rho, \rho ],
\end{equation}
and we fix a nondecreasing sequence $(w_k)_{k\ge 1}$, with $w_k\in(0,\rho]$, such that $\omega(w_k) \le c_\star  k^{-\alpha}$ for some constant $c_\star >0$.

Then, for each $k\in\N$, $k\ge1$, there exists a neural network $\phi_k\in\nN(2,2,K_{k,\alpha})$, with $K_{k,\alpha} = 2/(|\gamma|w_k^2)$, such that $\phi_k([0,1])\subseteq [0,1]$ and
\[
    \max_{x\in[0,1]} \left| x^2 - \phi_k(x) \right| \le \frac{c_\star}{|\gamma|} k^{-\alpha}.
\]
\end{lemma}

\begin{proof}
We replace $\sigma$ by $\tilde\sigma(t) := \sigma(t)-\sigma(x_0)$ so that $\tilde\sigma(x_0)=0$, and we shift the origin in $x_0$ so that $x_0=0$. Next, we define, for each $k\ge1$, the width-$2$ neural network
\[
    \Phi_k(x) := d_k \left(\sigma(w_k x)+\sigma(-w_k x)\right), \qquad \text{with} \quad d_k:=\tfrac{1}{\gamma w_k^2}.
\]
We notice that $\Phi_k\in\nN(2,1,K_{k,\alpha})$, with $K_{k,\alpha} = \frac{2}{|\gamma|w_k^2}$. By applying~\eqref{eq:even_modulus} with $h=w_k x$ (note that $|h|\le w_k\le \rho$ for $x\in[0,1]$), we have that
\[
    \sigma(w_k x)+\sigma(-w_k x) = \gamma (w_k x)^2 + r_k(x), \qquad \text{with} \quad |r_k(x)| \le \omega(|w_k x|) w_k^2 \,x^2 \le \omega(w_k) w_k^2 \, x^2.
\]
Multiplying by $d_k=1/(\gamma w_k^2)$ yields
\[
\Phi_k(x) = x^2 + \frac{r_k(x)}{\gamma w_k^2} \implies  \bigl|x^2-\Phi_k(x)\bigr| \le \frac{\omega(w_k)}{|\gamma|}x^2 \le \frac{\omega(w_k)}{|\gamma|}\le\frac{c_\star}{|\gamma|}k^{-\alpha}.
\]
We next define the clipping $\psi(t):=\mathrm{ReLU}(t)-\mathrm{ReLU}(t-1)\in\nN(2,1,1)$ and $\phi_k=\psi\circ\Phi_k\in\nN(2,2,K_{k,\alpha})$. Thus $\phi_k([0,1])\subset[0,1]$ and
\[
    \max_{x\in[0,1]} \left| x^2 - \phi_k(x) \right| \le \frac{c_\star}{|\gamma|} k^{-\alpha}.
\]
follows from the fact that $\psi$ is $1$-Lipschitz with $\psi(z)=z$ on $[0,1]$.
\end{proof}

\begin{remark}
Assumption~\eqref{eq:even_modulus} is a locally $C^2$ condition of $\sigma$ at $x_0$. No third and fourth derivatives are needed.
\end{remark}

\begin{corollary}\label{cor:weak_holder}
    Let $\sigma:\R\to\R$ be locally Lipschitz and $\alpha>0$. We assume that there exist $x_0\in\R$, $0<\rho\le1$, $\gamma\neq 0$, and $C_\beta>0$, with $\beta\in(0,2]$, such that
    \begin{equation}\label{eq:Ho}
        \left|\sigma(x_0+h)+\sigma(x_0-h)-2\sigma(x_0)-\gamma h^2\right| \le C_\beta|h|^{2+\beta} \qquad \text{for all } h \in[- \rho, \rho ].
    \end{equation}
    Then, for each $k\in\N$, $k>\rho^{-\frac{\beta}{\alpha}}$, there exists a neural network $\phi_k\in\nN(2,2,K_{k,\alpha})$, with $K_{k,\alpha} = \frac{2}{|\gamma|}k^\frac{2\alpha}{\beta}$, such that $\phi_k([0,1])\subseteq[0,1]$ and
    \[
        \max_{x\in[0,1]} \left| x^2 - \phi_k(x) \right| \le \frac{C_\beta}{|\gamma|} k^{-\alpha}.
    \]
\end{corollary}

\begin{proof}
    The proof follows from~\cref{lem:weak_modulus} by noticing that assumption~\eqref{eq:Ho} corresponds to assumption~\eqref{eq:even_modulus} with $\omega(t) = C_\beta t^\beta$ and by choosing $w_k = k^{-\frac{\alpha}{\beta}}$.
\end{proof}

\begin{remark}
If we take $\beta=2$ in~\cref{cor:weak_holder}, we get $w_k=k^{-\alpha/2}$ and we recover the same neural network construction of~\cref{lemma:approximation_x2_regular}. Thus, this result strictly weakens the regularity requirements on the activation function $\sigma$ while preserving the same approximation rate $k^{-\alpha}$.
\end{remark}

\begin{example} \label{example:1}
In each example below $\sigma$ is locally Lipschitz and satisfies~\eqref{eq:even_modulus} with some $\gamma\neq 0$ and a modulus $\omega\downarrow 0$, yet \emph{does not} satisfy Assumption \ref{ass:sigma_exp} required to apply~\cref{lemma:approximation_x2_regular}.

\medskip
\noindent\textbf{(A) $C^2$ but not $C^3$: a cubic cusp.}
We define
\[
\sigma_A(x):=x^2 + c\,|x|^3 \qquad (c\neq 0).
\]
Then, for all $h\in\R$, it holds that
\[
\sigma_A(h)+\sigma_A(-h)-2\sigma_A(0)=2h^2 + 2c|h|^3 = \gamma h^2 + (2c |h|) \,h^2,
\]
so~\cref{lem:weak_modulus} holds with $\gamma=2$ and $\omega(t)=2|c| t$. We notice that $\sigma_A\in C^2$ but not $C^3$ at $0$ (the third derivative jumps due to the modulus), hence a global $O(|x|^4)$ Taylor remainder does not exist.

\medskip
\noindent\textbf{(B) $C^{2,\beta}$ but not $C^3$: a H\"older bump.}
We fix $\beta\in(0,2)$ and define
\[
\sigma_B(x):=x^2 + c\,|x|^{2+\beta}\qquad (c\neq 0).
\]
Then, for all $h\in\R$, it holds that
\[
\sigma_B(h)+\sigma_B(-h)-2\sigma_B(0)=2h^2 + 2c|h|^{2+\beta} = \gamma h^2 + (2c|h|^{\beta}) \, h^2,
\]
so~\cref{lem:weak_modulus} holds with $\gamma=2$ and $\omega(t)=2|c| t^{\beta}$. If $\beta\in(0,1]$, then $\sigma_B\notin C^3$ at $0$; if $\beta\in(1,2)$, then $\sigma_B\in C^3$ but not $C^4$, and the third–order Taylor remainder is $O(|x|^{2+\beta})\not=O(|x|^4)$.

\medskip
\noindent\textbf{(C) $C^2$ with a logarithmic modulus (no power law).}
We define, for $x\in\R$,
\[
\sigma_C(x):=x^2\left(1+\frac{1}{\log \left(e/|x|\right)}\right),\qquad \sigma_C(0):=0.
\]
Then, for $h\in\R$, it holds that
\[
\sigma_C(h)+\sigma_C(-h)-2\sigma_C(0) = 2h^2 \left(1+\frac{1}{\log \left(\mathrm e/|h|\right)}\right) = \gamma h^2 + \frac{2}{\log(\mathrm e/|h|)} \, h^2.
\]
Thus~\cref{lem:weak_modulus} holds with $\gamma=2$ and a \emph{logarithmic} modulus $\omega(t)=|2/\log(\mathrm e/t)|\downarrow 0$. The function is $C^2$ at $0$, but it is not $C^{2,\beta}$ for any $\beta>0$ (and hence not $C^3$), and the remainder is not $O(|x|^{2+\beta})$ for any $\beta>0$, in particular not $O(|x|^4)$.

\medskip
\noindent\textbf{(D) Odd Lipschitz perturbations.}
Let $g$ be any locally Lipschitz \emph{odd} function (i.e. $g(-x)=-g(x)$) near $0$, and we consider
\[
\sigma_D(x):=x^2 + g(x).
\]
Then $g(h)+g(-h)=0$ for all $h$, so it holds that
\[
\sigma_D(h)+\sigma_D(-h)-2\sigma_D(0)=2h^2,
\]
and~\cref{lem:weak_modulus} holds with $\gamma=2$ and \emph{zero} modulus $\omega\equiv 0$. By choosing, for instance,
\[
g(x)=\lambda\,x|x|\qquad\text{or}\qquad g(x)=\lambda\,x\sin\left(\log(1/|x|)\right)\,\mathbf{1}_{\{|x|<\mathrm e^{-1}\}},
\]
we obtain locally Lipschitz activations that are not $C^2$ at $0$, so they cannot admit a third–order Taylor expansion with $O(|x|^4)$ remainder. Nevertheless, the even part still yields the exact quadratic coefficient required by~\cref{lem:weak_modulus}.
\end{example}

\section{Numerical validation of low regularity approximation} \label{sec:numerical_weaker}

In this section, we present numerical results that validate~\cref{lem:weak_modulus} and~\cref{cor:weak_holder} by using the two–neuron symmetric approximation with the activation functions $\sigma$ in~\cref{example:1}. Specifically, we consider
\[
\Phi_k(x) = \tfrac{1}{\gamma w_k^2}\bigl(\sigma(w_k x)+\sigma(-w_k x)-2\sigma(0)\bigr), \qquad \phi_k(x) = \psi \left(\Phi_k(x)\right), \quad \psi(t)= \relu(t)-\relu(t-1).
\]
Here $\gamma$ is the even–part quadratic coefficient at $0$ (for all our examples $\gamma=2$), and $w_k$ is the weight chosen to match the modulus $\omega$ in~\eqref{eq:even_modulus} as follows.
\begin{itemize}
    \item For power–law modulus $\omega(t)\asymp t^\beta$, we take weights $w_k=k^{-\alpha/\beta}$ (cases~\textbf{(A)}–\textbf{(B)}).
    \item For the logarithmic modulus $\omega(t)\asymp 1/\log(\mathrm e/t)$, we use $w_k=\mathrm{e}^{1-k^\alpha}$ (case~\textbf{(C)}).
    \item In the exact–cancellation case~\textbf{(D)}, any vanishing weights $w_k$ works (in our the experiment we use $w_k=k^{-\alpha/2}$).
\end{itemize}
In all experiments, we fix $\alpha=1$. In~\cref{tab:weak-lemma}, we report the following quantities: the \emph{unclipped} and \emph{clipped} errors; the predicted bound $\omega(w_k)/|\gamma|$; the range of $\Phi_k$ on $[0,1]$. Looking at~\cref{tab:weak-lemma}, we draw the following conclusions.
\begin{itemize}
    \item Cases (A)–(B) exhibit errors consistent with $k^{-\alpha}$ (since $w_k^\beta\asymp k^{-\alpha}$). The empirical errors halve approximately when $k$ doubles, consistent with the design $w_k\sim k^{-\alpha/\beta}$ and the bound $\omega(w_k)\sim w_k^\beta\sim k^{-\alpha}$.
    \item Case (C) shows the slower, logarithmic decay dictated by its modulus. The decay follows the slower logarithmic modulus; the growth of $d_k=1/(\gamma w_k^2)$ is correspondingly faster, which is visible in the increasingly narrow transition region of $\Phi_k$.
    \item Case (D) achieves machine precision due to exact cancellation of the odd perturbation in the even block. The odd perturbation cancels exactly in the even block, so $\Phi_k\equiv x^2$ on $[0,1]$ up to round-off.
\end{itemize}
\cref{fig:app_abs_A,fig:app_abs_B,fig:app_abs_C} show approximation and absolute errors for cases (A), (B) and (C) respectively.

\begin{table}[h]
  \centering
  \small
  \begin{tabular}{l|cc|ccccccc}
    \toprule
    Activation function & $k$ & $\alpha$ & $w_k$ & $d_k$ & $\|\Phi_k-x^2\|_\infty$ & $\|\phi_k-x^2\|_\infty$ & pred. & $\min\Phi_k$ & $\max\Phi_k$ \\
    \midrule
    \textbf{A)} $x^2 + |x|^3$ & 8  & 1.0 & 0.125   & 32.0    & 0.1250 & 0.1056 & 0.1250 & 0.0 & 1.1250 \\
                      & 16 & 1.0 & 0.0625  & 128.0   & 0.0625 & 0.0572 & 0.0625 & 0.0 & 1.0625 \\
                      & 32 & 1.0 & 0.0313 & 512.0   & 0.0313 & 0.0298 & 0.0312 & 0.0 & 1.0313 \\
                      & 64 & 1.0 & 0.0156 & 2048.0  & 0.0156 & 0.0152 & 0.0156 & 0.0 & 1.0156 \\
    \midrule
    \textbf{B)} $x^2 + |x|^{2+0.7}$ & 8  & 1.0 & 0.0513 & 190.21  & 0.1250 & 0.1071 & 0.1250 & 0.0 & 1.1250 \\
                           & 16 & 1.0 & 0.0190 & 1378.20   & 0.0625 & 0.0576 & 0.0625 & 0.0 & 1.0625 \\
                           & 32 & 1.0 & 0.0071 & 9986.16   & 0.0313 & 0.0299 & 0.0313 & 0.0 & 1.0313 \\
                           & 64 & 1.0 & 0.0026 & 72357.6   & 0.0156 & 0.0153 & 0.0156 & 0.0 & 1.0156 \\
    \midrule
    \textbf{C)} $x^2 \left(1 + \tfrac1{\log(\mathrm e/|x|)}\right)$ & 8  & 1.0 & 0.0009 & 601302.1     & 0.1250 & 0.1103 & 0.1250 & 0.0 & 1.1250 \\
                                                     & 16 & 1.0 & $3.059 \cdot 10^{-7}$ & $53 \cdot 10^{13}$     & 0.0625 & 0.0587 & 0.06250 & 0.0 & 1.0625 \\
                                                     & 32 & 1.0 & $3.4 \cdot 10^{-14}$ & $4.2 \cdot 10^{26}$ & 0.0313 & 0.0303 & 0.0313 & 0.0 & 1.0313 \\
                                                     & 64 & 1.0 & $4.3 \cdot 10^{-28}$ & $2.6 \cdot 10^{54}$ & 0.01563 & 0.0154 & 0.0156 & 0.0 & 1.0156 \\
    \midrule
    \textbf{D)} $x^2 + x|x|$ & 8  & 1.0 & 0.3536 & 4.0000  & $3.3 \cdot 10^{-16}$ & $3.3 \cdot 10^{-16}$ & 0.0 & 0.0 & 1.0 \\
                                      & 16 & 1.0 & 0.25               & 8.0                & 0.0                    & 0.0                    & 0.0 & 0.0 & 1.0 \\
                                      & 32 & 1.0 & 0.1768 & 16 & $3.3 \cdot 10^{-16}$ & $3.3 \cdot 10^{-16}$ & 0.0 & 0.0 & 1.0 \\
                                      & 64 & 1.0 & 0.125              & 32.0               & 0.0                    & 0.0                    & 0.0 & 0.0 & 1.0 \\
    \bottomrule
  \end{tabular}
    \caption{Numerical validation for the weak-regularity \cref{lem:weak_modulus} (all runs with $\alpha=1$ and $\gamma=2$). The column "pred." is the theoretical bound $\omega(w_k)/|\gamma|$.}
    \label{tab:weak-lemma}
\end{table}

\begin{figure}[h]
  \centering
  \begin{subfigure}[t]{0.45\textwidth}
  \includegraphics[width=\linewidth]{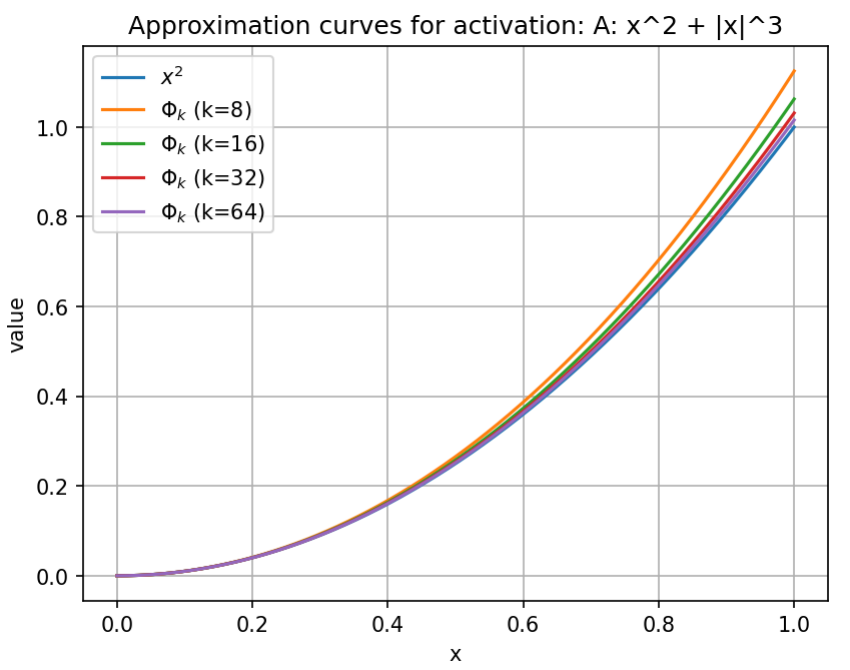}
  \end{subfigure}\hfill
  \begin{subfigure}[t]{0.45\textwidth}
  \includegraphics[width=\linewidth]{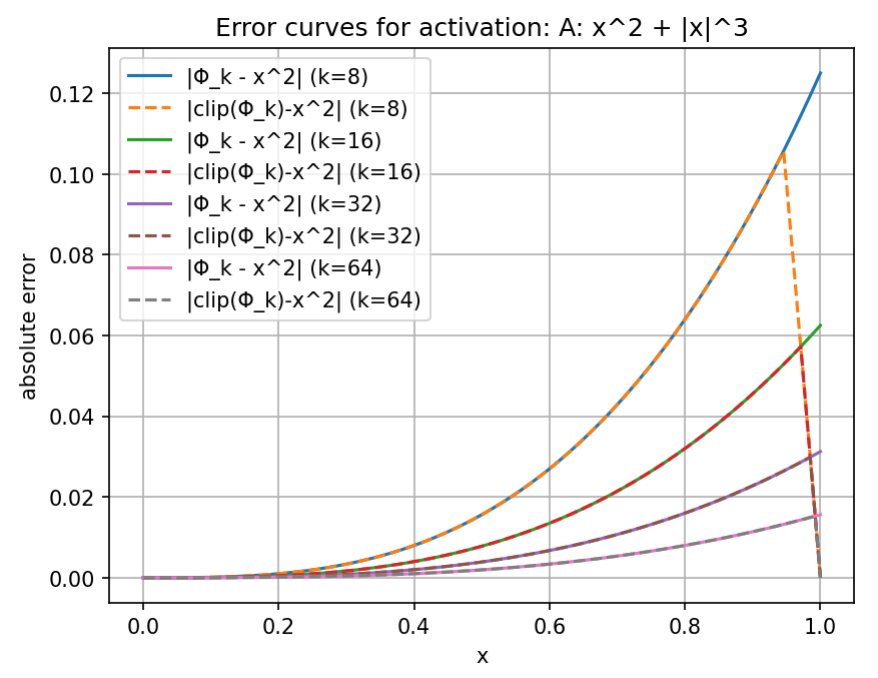}
  \end{subfigure}
  \caption{Approximation (left) and absolute error (right) for case (A). Solid: $\Phi_k$; dashed: $\phi_k=\mathrm{clip}(\Phi_k)$.}
  \label{fig:app_abs_A}
\end{figure}

\begin{figure}[h]
  \centering
  \begin{subfigure}[t]{0.45\textwidth}
  \includegraphics[width=\linewidth]{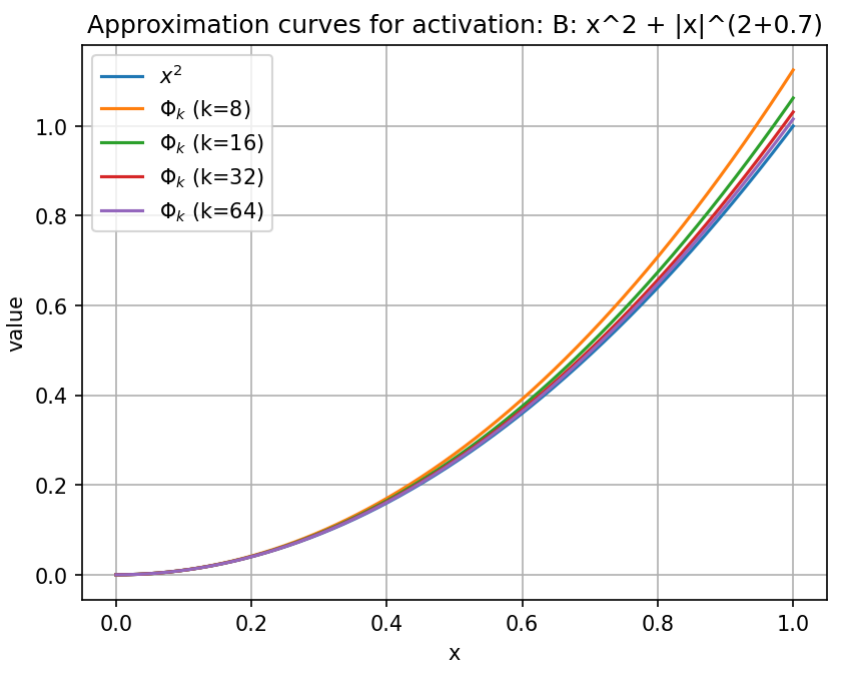}
  \end{subfigure}\hfill
  \begin{subfigure}[t]{0.45\textwidth}
  \includegraphics[width=\linewidth]{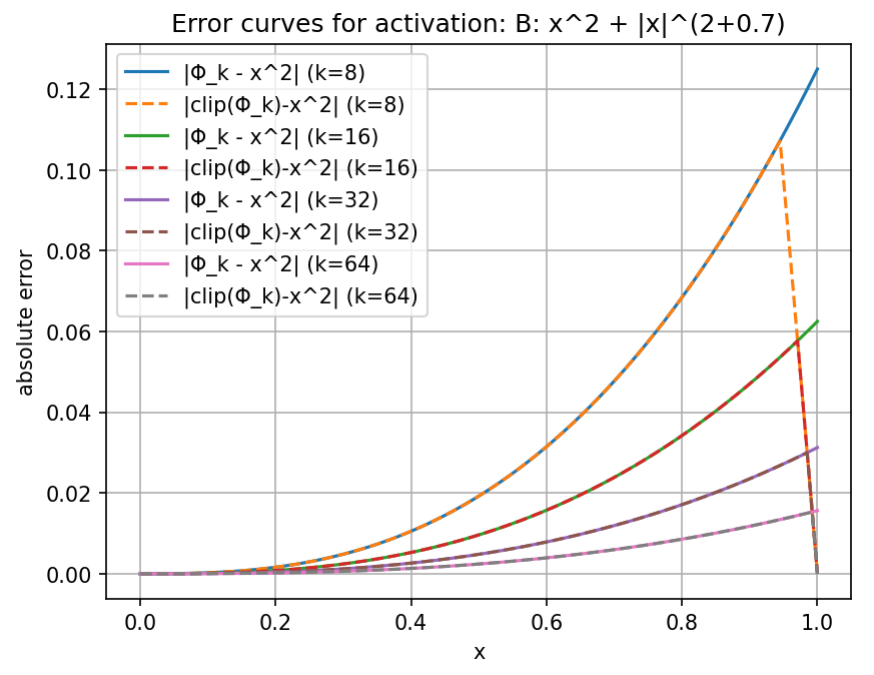}
  \end{subfigure}
  \caption{Approximation (left) and absolute error (right) for case (B). Solid: $\Phi_k$; dashed: $\phi_k=\mathrm{clip}(\Phi_k)$.}
  \label{fig:app_abs_B}
\end{figure}

\begin{figure}[h]
  \centering
  \begin{subfigure}[t]{0.45\textwidth}
  \includegraphics[width=\linewidth]{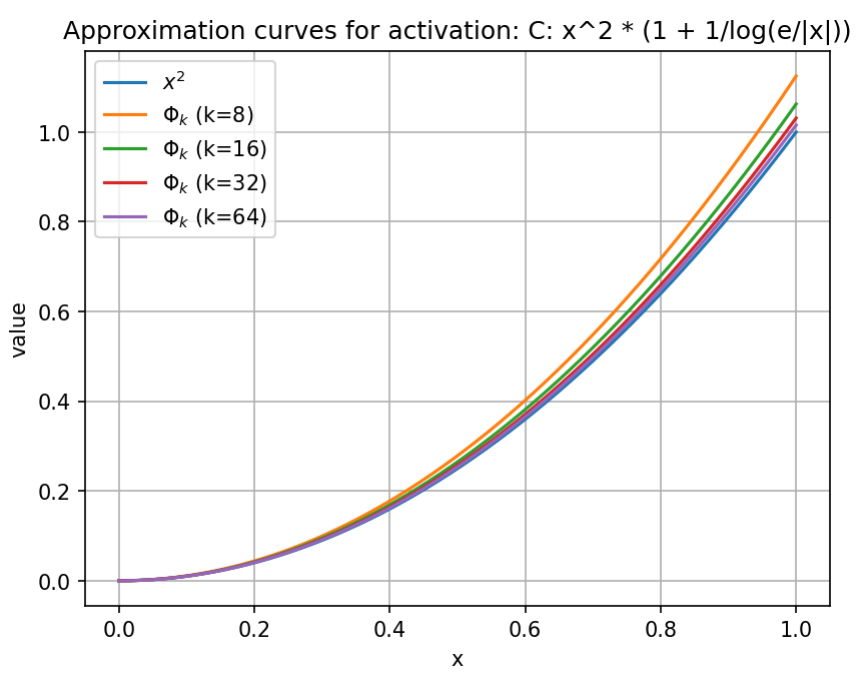}
  \end{subfigure}\hfill
  \begin{subfigure}[t]{0.45\textwidth}
  \includegraphics[width=\linewidth]{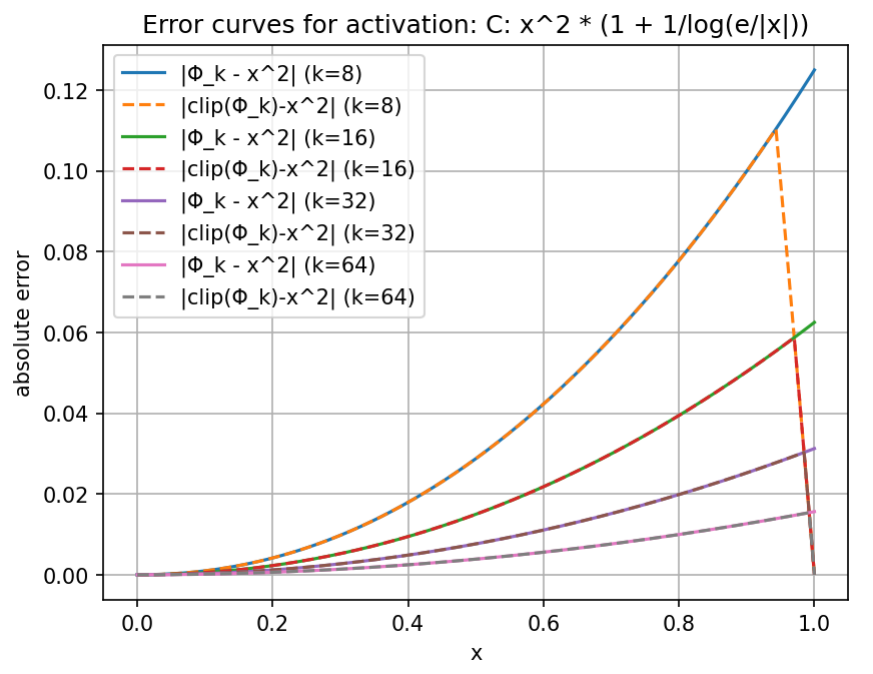}
  \end{subfigure}
  \caption{Approximation (left) and absolute error (right) for case (C). Solid: $\Phi_k$; dashed: $\phi_k=\mathrm{clip}(\Phi_k)$.}
  \label{fig:app_abs_C}
\end{figure}


\clearpage

\bibliographystyle{cas-model2-names}




\end{document}